\PassOptionsToPackage{unicode}{hyperref}
\PassOptionsToPackage{hyphens}{url}
\PassOptionsToPackage{dvipsnames,svgnames,x11names}{xcolor}
\documentclass[12pt]{article}

\usepackage{amsmath,amssymb,amsthm,amsfonts,amstext,amsbsy,amscd}

\usepackage[T1]{fontenc}
\usepackage[utf8]{inputenc}
\usepackage{textcomp} % provide euro and other symbols

\usepackage[]{natbib}
\usepackage{bookmark}

\IfFileExists{xurl.sty}{\usepackage{xurl}}{} % add URL line breaks if available
\hypersetup{
  pdftitle={Title},
  pdfauthor={Author 1; Author 2},
  pdfkeywords={3 to 6 keywords, that do not appear in the title},
  colorlinks=true,
  linkcolor={blue},
  filecolor={Maroon},
  citecolor={Blue},
  urlcolor={Blue},
  pdfcreator={LaTeX via pandoc}}

\usepackage{units}
\usepackage{hyperref}

\usepackage{tikz}
\usetikzlibrary{arrows.meta, positioning}
\numberwithin{equation}{section}
\usepackage{xcolor}
\usepackage{enumerate}
\usepackage{algorithm}
\usepackage{algorithmic}
\usepackage{dsfont}
\usepackage{float}
\usepackage{booktabs} % For better line separators in tables
\usepackage{graphicx}
\usepackage{subcaption}
\usepackage{comment}
\usepackage{amstext,amsbsy,amscd}

\newtheorem{satz}{Satz}[section] % Base counter resets per section
\newtheorem{definition}[satz]{Definition}
\newtheorem{theorem}[satz]{Theorem}

\newtheorem{lemma}[satz]{Lemma}
\newtheorem{proposition}[satz]{Proposition}

\numberwithin{equation}{section}
\numberwithin{figure}{section}
\numberwithin{table}{section}
\numberwithin{algorithm}{section}

\newtheorem{remark}[satz]{Remark}
\newtheorem{example}[satz]{Example}
\DeclareMathOperator{\E}{{\mathbb E}}

\DeclareMathOperator{\R}{{\mathbb R}}

\DeclareMathOperator{\PP}{{\mathbb P}}

\DeclareMathOperator{\trace}{trace}

 \DeclareMathOperator{\rank}{rank}

 \DeclareMathOperator{\Id}{Id}
\DeclareMathOperator{\argmin}{argmin}
\DeclareMathOperator{\Var}{Var}

\providecommand{\eps}{\varepsilon}
\renewcommand{\phi}{\varphi}

\renewcommand{\theta}{\vartheta}
\renewcommand{\subset}{\subseteq}

\renewcommand{\cdot}{{\scriptstyle \bullet} }

\providecommand{\abs}[1]{\lvert #1 \rvert}
\providecommand{\norm}[1]{\lVert #1 \rVert}

\providecommand{\babs}[1]{{\Bigl\lvert #1 \Bigr\rvert}}
\providecommand{\scapro}[2]{\langle #1,#2 \rangle}

\providecommand{\floor}[1]{\lfloor #1 \rfloor}

\renewcommand{\le}{\leqslant}
\renewcommand{\ge}{\geqslant}

\begin{document}

\def\spacingset#1{\renewcommand{\baselinestretch}%
{#1}\small\normalsize} \spacingset{1}

%%%%%%%%%%%%%%%%%%%%%%%%%%%%%%%%%%%%%%%%%%%%%%%%%%%%%%%%%%%%%%%%%%%%%%%%%%%%%%
% \address[A]{Institute of Mathematics and School of Business and Economics,
%  Humboldt-Universität zu Berlin \printead[presep={,\ }]{e1}}

% \address[B]{Institute of Mathematics,
%  Humboldt-Universität zu Berlin\printead[presep={,\ }]{e2}}

\title{\bf Early Stopping for Regression Trees}
\author{Ratmir Miftachov\thanks{Institute of Mathematics, School of Business and Economics, Humboldt-Universität zu Berlin, Germany}\and Markus Reiß\thanks{Institute of Mathematics, Humboldt-Universität zu Berlin, Germany}}  
\maketitle
\bigskip
\begin{abstract}
We develop early stopping rules for growing regression tree estimators. The fully data-driven stopping rule is based on monitoring the global residual norm. The best-first search and the breadth-first search algorithms together with linear interpolation give rise to generalized projection or regularization flows. A general theory of early stopping is established. Oracle inequalities for the early-stopped regression tree are derived without any smoothness assumption on the regression function, assuming the original CART splitting rule, yet with a much broader scope. The remainder terms are of smaller order than the best achievable rates for Lipschitz functions in dimension $d\ge 2$. In real and synthetic data the early stopping regression tree estimators attain the statistical performance of cost-complexity pruning while significantly reducing computational costs.
\end{abstract}

\noindent%
\vfill

\newpage
\spacingset{1.4}

%%%%%%%%%%%%%%%%%%%%%%%%%%%%%%%%%%%%%%%%%%%%%%
%%%% Main text entry area:
\section{Introduction}
%Motivation of the tree:
Classification and regression trees (CART) form a class of very popular statistical learning algorithms. Initially introduced by \citet{breiman1984classification},  CART serve as the foundation for more advanced machine learning techniques, including 
random forests \citep{breiman2001random}, XGBoost \citep{chen2016xgboost}, and others. Remarkably, many of the winning solutions for tabular data competitions on Kaggle.com rely on advanced variations of Breiman's decision tree algorithm. Because of their clear interpretability and their stability, practitioners still often prefer the basic CART methods to these more complex algorithms.

% bagged decision trees \citep{breiman1996bagging}, random forests \citep{breiman2001random}, and, notably, boosted decision trees like AdaBoost \citep{freund1997decision}, XGBoost \citep{chen2016xgboost}, CatBoost \citep{dorogush2018catboost}, and LightGBM \citep{ke2017lightgbm}

%Research Gap
We focus on regression trees, which are nonparametric estimators constructed by an iterative and greedy data-fitting algorithm. Typically, iterative estimators require a data-adaptive choice of the iteration number to efficiently prevent from over- and underfitting. The current folklore in practice for decision trees is  \textit{post-pruning}, as initially proposed by \citet{breiman1984classification}. This bottom-up method involves growing the tree to its full depth, which interpolates the data, and subsequently cutting it back iteratively to an optimal subtree using cross-validation. Oracle-type inequalities on the pruned tree have been established by \citet{gey2005model}. In recent work on asymptotic properties, trees are generally grown to their maximum possible depth \citep{biau2012analysis, scornet2015consistency}. Alternatively, the depth is controlled indirectly by limiting the number of observations in the terminal nodes \citep{breiman1984classification} or growing the tree until a prespecified depth is reached \citep{klusowski2023large}. These techniques are commonly referred to as \textit{pre-pruning} in the applied decision tree literature. The principle of early stopping as implicit regularization is, however, a cornerstone of modern machine learning \citep{goodfellow2016deep}.

%XOR
A common heuristic is to halt tree growth if the decrease in impurity (see Section \ref{sec:cart_algo} for details) between two subsequent tree depths falls below a certain threshold \citep{breiman1984classification, hastie2017elements}, a local rule which is used in software like scikit-learn \citep{pedregosa2011scikit}. 
The threshold should be chosen such that, with high probability, no further splits are conducted in the pure noise case (i.e., when $f$ is constant on the node), thus it should have the order of the noise variance $\sigma^2$.
As is well known, this rule fails on the extreme test case of the XOR problem, illustrated in Figure \ref{fig:XOR}. The greedy CART algorithm based on the local impurity gain terminates at the root node, without conducting any split. It predicts solely the unconditional mean of the response. In contrast, our proposed early stopping technique overcomes this shortcut, since it forces the algorithm to split until the residuals are small enough.

%\citep{mazumder2023convergence} a possible citation for XOR, but not enough space.

% \begin{figure}[tbp]
%     \centering
%     \includegraphics[width=0.9\textwidth]{JASA/xor_example_results.png}
%     \caption{Left: true function $f(x)= \operatorname{sign}(x_1)\operatorname{sign}(x_2)$ with $X_i \overset{\mathrm{iid}}{\sim} U(-1, 1)^{2}$ and i.i.d. noise $\varepsilon_i \sim N(0,0.01)$. Center: semi-global early stopped regression tree. Right: pre-pruned regression tree with a minimum impurity decrease of $0.01$.}
%     \label{fig:XOR}
% \end{figure}

% \begin{figure}[tbp]
%     \centering
%     \includegraphics[width=0.9\textwidth]{xor_example_results_adjusted.png}
%     \caption{Left: true function $f(x)= \operatorname{sign}(x_1)\operatorname{sign}(x_2)$ with $X_i \overset{\mathrm{iid}}{\sim} U(-1, 1)^{2}$ and i.i.d. noise $\varepsilon_i \sim N(0,0.01)$. Center: semi-global early stopped regression tree. Right: local pre-pruned regression tree with a minimum impurity decrease of $0.01$.}
%     \label{fig:XOR}
% \end{figure}

%change first plot to Y, instead of f(x):

\begin{figure}[tbp]
    \centering
    \includegraphics[width=0.9\textwidth]{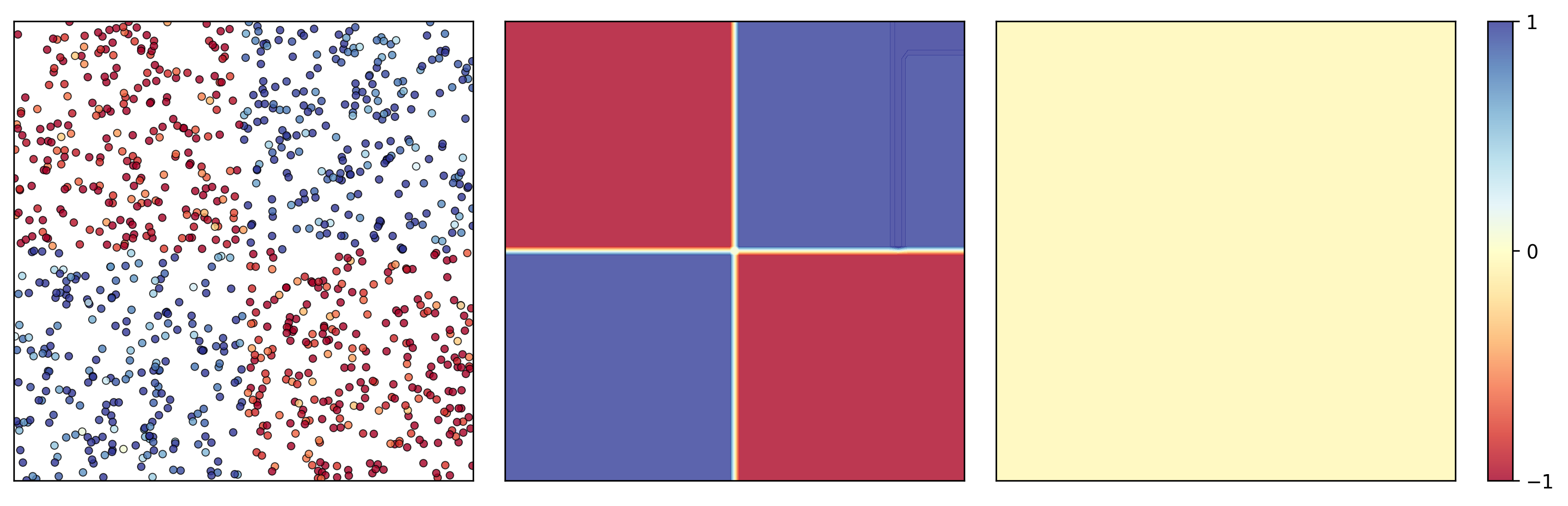}
    \caption{Left: $Y_i=f(X_i)+\varepsilon_i \text{ with } f(x)=\operatorname{sign}(x_1)\operatorname{sign}(x_2)$ and $\varepsilon_i\sim N(0,0.1)$, where $X_i\sim U(-1,1)^2$. Center: semi-global early stopped regression tree (with 11 terminal nodes). Right: local pre-pruned regression tree with a minimum impurity decrease of $0.1$.}
    \label{fig:XOR}
\end{figure}

The current state of the art on tuning the decision tree depth, the post-pruning, is computationally very costly. We consider a fundamentally different approach in which the decision tree is grown iteratively and stopped early using an entirely data-driven stopping rule, usually far before the tree is fully grown. Unlike previous pre-pruning (or early stopping) methods, our stopping rule monitors the global residual norm, as for the discrepancy principle for inverse problems \citep{englregularization} and its statistical applications for linear estimators \citep{blanchard2012discrepancy}. 
%In our analysis, we shall derive oracle-type inequalities for the prediction error, compared to an oracle iteration number where approximation and stochastic error are ideally balanced. In our simulations and real data analysis below, this approach significantly reduces computational costs by a factor of 10 to 20. At the same time, its prediction performance is comparable to standard pruning algorithms and in simulations it comes close to that of the oracle. 
%Moreover, overfitting is clearly prevented so that the resulting trees are well interpretable. [\MR weg: A high-level comparison of our work to competing methods is given in Table \ref{tab:method_comparison}.]

%How do we fill this research gap?
Let us list our main contributions:
\begin{itemize}
\item We interpret the regression tree iterates as a special instance of piecewise constant estimators on a sequence of refining partitions. Using linear interpolation between consecutive iterates, we put them in the framework of \textit{generalized projection flows}. For this general concept, which also encompasses other popular smoothing estimators like ridge regression or gradient flow, we build a general theory of early stopping based on the residual norm. We derive an $\omega$-wise, that is non-probabilistic, error decomposition and an oracle inequality, comparing to the theoretically best estimator along the flow.
\item We specify the error terms for the breadth-first and best-first regression tree building, which refine partitions globally and locally, respectively. Given the globally defined residual norm, they give rise to the global and semi-global early stopping methods for trees. A main mathematical challenge is the bound for the cross term for the data-driven splitting in CART. The final oracle inequalities show that the error due to data-driven early stopping is typically smaller than the oracle error so that the early stopping estimator becomes statistically adaptive.
     
\item Based on a nearest neighbour estimate of the noise level, we provide ready-to-use algorithms and apply them to standard data sets and in simulations. We implement our stopping algorithm in a Python library, see \citet{ziebell2025earlystopping}. In practice,  early stopping reduces the computational run time by a factor of 10 to 20 while being on par with the performance of the state-of-the-art pruning methods. The early stopped regression tree usually comes with a small number of nodes, thus leading to a sparse representation and good interpretability. 

\end{itemize}

%Related Literature. How are we different?
This work bridges the decision tree and random forest literature with the early stopping literature based on the discrepancy principle. 
%In the following, we introduce closely related results and differentiate our work from them.
%Early Stopping main references
Recent advances in the iterative early stopping literature have evolved from statistical inverse problems such as \citet{blanchard2018early, blanchard2018optimal} for linear estimators. \citet{stankewitz2024early} has applied the early stopping methodology for $L^2$-boosting in the setting of high-dimensional linear models. Our general framework unifies arguments from these papers and opens up an even wider scope of applications. Conceptually, our projection flow approach is inspired by the general theory for histogram regression estimators of \citet{nobel1996histogram}.

% Further work on early stopping in nonparametric regression includes \citet{celisse2021analyzing}, which investigated kernelized spectral filter learning methods. Early stopping bounds for nonparametric regression with reproducing kernel Hilbert spaces have been studied in \citet{raskutti2014early}. %\citet{jahn2024early} focuses on theoretical results for early stopping in convolutional neural networks. 
%Early stopping for $L^2$ boosting in linear models was explored by \citet{buhlmann2003boosting, buhlmann2006boosting}, with recent advancements by \citet{stankewitz2024early}.

% Depth assumption in tree literature
Consistency for the random forest under Breiman's greedy splitting is proven by \citet{scornet2015consistency} under the additive model assumption; yet, convergence rates are not provided. \citet{chi2022asymptotic} prove consistency for the original Breiman algorithm under further assumptions in a high-dimensional setting with a polynomial convergence rate in $n$. Subsequently, under the assumption of an additive model, \citet{klusowski2023large} achieve a logarithmic rate for the regression tree. They control the tree depth using a depth parameter, but do not provide data-driven guidelines for choosing the tree depth. 
%In contrast to the previously mentioned paper, the leading term in our oracle inequality for the early stopped tree estimator is of order $n^{-1/2}$ and is general on the functional class.

%Literaturüberblick!
%Literatur zu daten-unabh. splitting:
The literature has tried to reduce the splitting complexity by working with simplified versions of the regression tree. 
%These simplified splitting versions are commonly referred to as purely random forests. 
For the purely uniform random forest, \citet{genuer2012variance} provides a convergence rate of $n^{-2/3}$ under Lipschitz conditions in dimension one. Further convergence rates are given in \citet{biau2012analysis} and \citet{arlot2014analysis} for the centered forest under Hölder class assumptions for $d\geq 1$. Subsequently, \citet{klusowski2021sharp} improves the rate of \citet{biau2012analysis} for the centered forest, yet not attaining the minimax optimal rate for Lipschitz functions and $d>1$. Further work on the Mondrian forest achieves minimax optimality under appropriate tuning parameter selection \citep{cattaneo2023inference}. Here, we show that the dominant term in our oracle inequality for the data-independent case is negligible with respect to these rates. In further research, the $\alpha$-fraction constraint was introduced in \citet{wager2015adaptive} and utilized in consistency results of recent work in \citet{athey2019generalized} and \citet{wager2018estimation} in 'honest trees'. \citet{cattaneo2024convergence} develop a theory for oblique decision trees, where the splits are not axis-aligned. \citet{cattaneo2023inference} show that the Mondrian random forest is minimax optimal. An extensive overview of the convergence rates for various random forest versions is given in \citet{zhang2024adaptive}. 
% The dyadic CART was initially introduced by \citet{donoho1997cart}. \citet{chaudhuri2023cross} contribute to the theoretical understanding of cross-validation for the dyadic CART.
% [\MR unify with above ]

%Structure of paper
We proceed as follows. In Section \ref{sec:stats}, we define the statistical setting and introduce the regression tree algorithm. Section \ref{sec:framework} provides the theory in a unifying framework, beginning with refined orthogonal projections and proceeding to the generalized projection flows in Section \ref{sec:gen_flow}. First, Section \ref{SecOI} provides a general error decomposition and oracle inequality. Then, specifying to global and semi-global early stopping for regression trees, our main results, oracle inequalities under the original CART algorithm, are established in Section \ref{sec:main}.
In Section \ref{sec:implementation}, we elaborate on the implementation of our algorithms and apply them to standard datasets. In Section \ref{sec:sim}, we compare the early-stopped estimator with post-pruning in simulations. Section \ref{sec:discussion} concludes the paper with a short discussion. All proofs and further simulation results are delegated to the Appendix.

%%%%%%%%%%%%%%%%%%%%%%%%%%%%%%%%%%%%%%%RM likes it and wants to have it in the ARXIV version or presentation:
%%%%%%%%%%%%%%%%%%%%%%%%%%%%%%%%%%%%%%%
% \begin{table}[tbp]
% \centering
% \footnotesize
% \begin{tabular}{lcccc}
% \toprule
% Method & Interpretable & Time & Theoretical results & Tuning required \\
% \midrule
% Deep tree & No & Medium & On consistency & No \\
% Post-pruning & Yes & Slow & Not many & Yes, CV \\
% Pre-pruning & Yes & Medium & No & Yes, CV \\
% Early stopping & Yes & Fast & This paper & No \\
% \bottomrule
% \end{tabular}
% \caption{Conceptual comparison of state-of-the art methods. Cross-validation (CV) refers to, e.g., 5-fold cross-validation.}
% \label{tab:method_comparison}
% \end{table}

\section{Regression trees}
\label{sec:stats}

\subsection{Regression model}

Let the $d$-dimensional covariates be denoted as $X_i=(X_{i,1},\ldots,X_{i,d})^{\top}\in\mathcal{X}\subseteq\mathbb{R}^{d}$, with real-valued response variables $Y_i\in\mathbb{R}$ for $i=1,\ldots,n$. We assume $(X_i, Y_i) \overset{\mathrm{iid}}{\sim} \mathbb{P}$ for some distribution $\mathbb{P}$.
 Consider the  nonparametric regression setting
\begin{align}
    Y_i=f(X_i)+\varepsilon_i,\quad i&=1,\ldots,n,  \text{ with } \E[\varepsilon_i\,|\,X_i]=0,\, \Var(\varepsilon_i\,|\,X_i)=\sigma^2.\label{EqRegr}
\end{align}
The target of estimation is the regression function  $f(x)=\E[Y|X=x]$. We assume throughout that the design points $X_1,\ldots,X_n$ are a.s. distinct.
For functions $g:\R^d\to\R$ we consider their empirical norm
$\norm{g}_n:=\Big(\frac1n\sum_{i=1}^ng(X_i)^2\Big)^{1/2}$,
the corresponding scalar product $\scapro{\cdot}{\cdot}_n$ and the associated $n$-dimensional function space $L_n^2$, obtained by identifying $g_1,g_2:\R^d\to\R$ with $\norm{g_1-g_2}_n=0$, that is coinciding on the design. The vectors $(Y_i)$, $(\eps_i)$ can  be lifted from $\R^n$ to $L_n^2$ via $Y(X_i)=Y_i$, $\eps(X_i)=\eps_i$ and, similarly, $g\in L^2_n$ is encoded by $(g(X_i))_{i=1,\ldots,n}\in\R^n$. In this sense, $L_n^2$ is identified with $\R^n$.

\subsection{Regression tree algorithm}
\label{sec:cart_algo}

We follow the classical algorithm introduced by \citet{breiman1984classification}, a greedy tree-growing procedure where the nodes identify subsets $A$ of $\R^d$.
The CART algorithm starts with a parent node $A=\mathcal{X}$, recursively partitioned into non-overlapping $d$-dimensional (hyper-)rectangles. The partitioning begins by selecting a coordinate $x_j$, $j=1,\ldots,d$, and a threshold $c\in\R$, dividing $A=\mathcal{X}$ into the (in coordinate $j$) left child node $A_L = \{x \in A : x_j < c\}$ and the right child node $A_R = \{x \in A : x_j \geq c\}$. Here, all sets in a partition are required to contain at least one design point $X_i$, sets containing only one $X_i$ are no longer split. The child nodes $A_L,A_R$ then serve as parent nodes for subsequent iterations, continuing the recursive partitioning process. At each instance, the {\it terminal nodes} (nodes without children) define a {\it partition} of $\R^d$, which is refined as the tree grows.

The CART algorithm is characterized by a data-driven splitting rule. Consider a generic parent node $A$ and let $n_A$ be its local sample size, that is the number of covariates $X_i$ with $X_i\in A$. We define the splitting criterion or impurity gain as
\begin{equation*}
L_A(j, c) = R^2(A) - \frac{n_{A_L}}{n_A} R^2(A_L) - \frac{n_{A_R}}{n_A} R^2(A_R), 
\end{equation*}
where the impurity for regression is set as the average squared residuals within a node (also called \textit{impurity})
\begin{equation*}
    R^2(A) = \frac{1}{n_A} \sum_{i:X_i \in A}\left(Y_i-\bar{Y}_{A}\right)^2,
\end{equation*}
with $\bar{Y}_{A} = \frac{1}{n_A} \sum_{i:X_i \in A} Y_i$ being the node average. The maximum impurity gain and the corresponding splitting rule are given by
\begin{equation}
    \mathbb{IG}(A)= \max_{(j,c)} L_A(j, c), \quad  (j^*, c^*) \in \underset{(j,c)}{\operatorname{argmax}}\ L_A(j, c),  \label{eq:opt_split}
\end{equation}
respectively. The regression tree is thus based on a greedy minimization of the training error in each split. Other splitting variants include splitting each node at the empirical median \citep{duroux2018impact} or a randomly chosen threshold \citep{genuer2012variance}, among others.

%Abgrenzen zu 'early stopping/pre-pruning' aus applied ML 
Often, a stopping rule for the tree depth is applied to find a balance between underfitting and overfitting of the piecewise constant estimator based on the partition defined by the terminal nodes. 
Common approaches that are referred to as pre-pruning (or early stopping) in the applied literature include growing the tree until it reaches a pre-specified depth \citep{klusowski2021sharp} or stopping the splitting process at a node when it contains a single observation, whichever occurs first (see, e.g., \citet{klusowski2023large}). Another indirect way to control tree depth is by limiting the maximum number of observations allowed in the terminal nodes \citep{breiman1984classification}. In contrast to these strategies, \citet{breiman1984classification} introduced cost-complexity post-pruning, a popular bottom-up approach that we use as a benchmark in our empirical analysis (see Section \ref{sec:pruning}). Note that also standard sequential model selection procedures like Lepski's method start in the low bias situation and then gradually increase the bias, see e.g. \citet{lepski1997optimal}, and thus require a bottom-up approach in our tree framework.
In this work, we introduce early stopping based on the discrepancy principle, which fundamentally differs from the previous stopping methods. It is a fully data-adaptive top-down approach without relying on cross-validation.

%Intro of 'global' and 'semi-global':
There are at least three common sequential ordering mechanisms to build a decision tree. This defines an order to which we apply the residual-based early stopping rule. Initially, the original CART algorithm is based on a depth-first search principle to grow the tree. However, this approach is statistically unsuitable for implementing our stopping condition, as Breiman's tree grows each branch to a pre-specified depth. In contrast, we propose combining two distinct ordering mechanisms, the breadth-first search and the best-first search, with our data-driven stopping rule. We call these approaches \textit{global} early stopping and \textit{semi-global} early stopping, respectively.
In the global stopping procedure, all nodes at a given generation (i.e., level or depth) are split simultaneously before progressing to the next generation.  On the other hand, the semi-global procedure splits one node at a time, selecting the next node to split based on the largest impurity gain, $\mathbb{IG}(A)$, across all partitions of the tree. Both stopping methods are illustrated in Figure \ref{fig:trees_plots}. The following chapter presents a unified mathematical framework for these two proposed stopping variants.

%computational stuff:
% Note that Python’s \texttt{scikit-learn} function \texttt{DecisionTreeRegressor()} uses the depth-first search under the default specification, where the maximal depth of the tree is specified in advance. (via the \texttt{max\_depth} parameter).
% (Best first search) It was initially introduced by \citet{friedman2001greedy} for boosted decision trees, and, most notably, the gradient boosting framework LightGBM \citep{ke2017lightgbm} is based on it. A further heuristic investigation is conducted by Shi (2008). 

% If you specify max_leaf_nodes > 0, then the best-first search is used in \texttt{DecisionTreeRegressor()}.
% In case of no early stopping, all three procedures result in the same tree. 

\begin{figure}[tbp]
    \centering
    \begin{subfigure}[b]{0.4\textwidth}
        \centering
        \hspace{-2cm}
  \includegraphics[width=\textwidth]{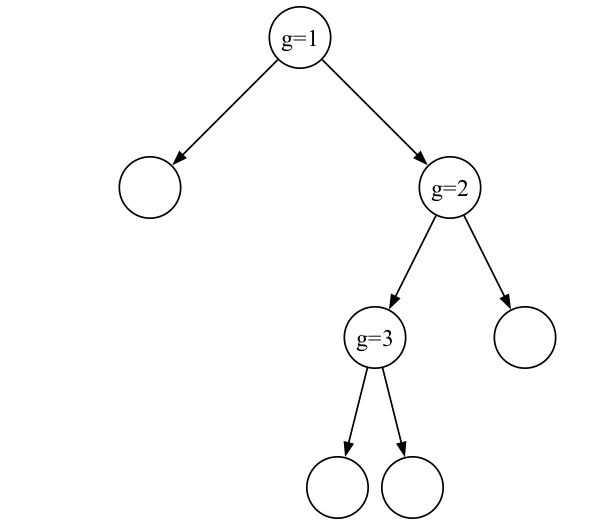}
    \end{subfigure}
    \hspace{0cm}
    \begin{subfigure}[b]{0.4\textwidth}
        \centering
  \includegraphics[width=\textwidth]{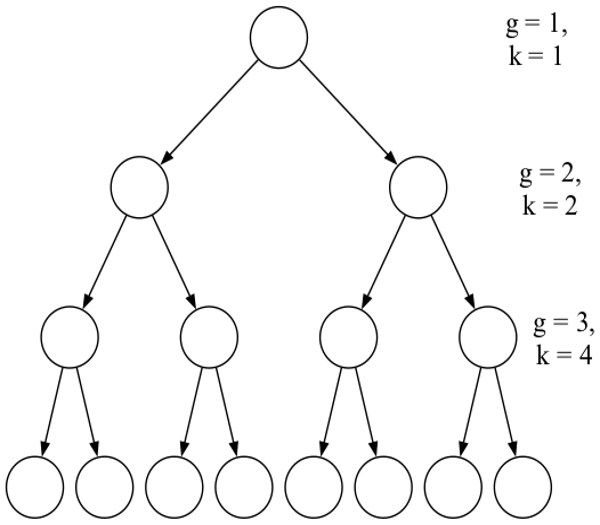}
    \end{subfigure}
    \caption{Exemplary semi-global (left) and global (right) tree structure after three splitting iterations/generations. The global stopped tree relies on  breadth-first search, whereas the semi-global stopped tree is based on best-first search. The generation/iteration is denoted by $g$, and the number of terminal nodes is denoted by $k$. }
    \label{fig:trees_plots}
\end{figure}

\newpage
\section{A unifying framework}
\label{sec:framework}

\subsection{Refining orthogonal projections}
Let us consider general partitions $P=\{A_1,\ldots,A_m\}$ of $\mathcal{X}\subset\R^d$ such that $\mathcal{X}=\bigcup_{k=1}^mA_k$, the $A_k$ are pairwise disjoint and contain each at least one design point $X_i$, which requires $m\le n$. Such a partition $P$ generates the $m$-dimensional subspace
\begin{equation*}
 V_P=\{f:\mathcal{X}\to\R : f|_{A_k}=\text{constant for all }k=1,\ldots,m\}\subset L_n^2
\end{equation*}
and the orthogonal projection $\Pi_P:L^2_n\to V_P$ onto this subspace. Thus,
\begin{equation*}
\Pi_Pf(x)=\bar f_{A_k}:=\frac{1}{n_{A_k}}\sum_{i:X_i\in A_k}f(X_i)\text{ for } x\in A_k
\end{equation*}
yields the  average of $f(X_i)$ for the $X_i$ in the set $A_k$ containing $x$. As an orthogonal projection $\Pi_P$ is selfadjoint, positive semi-definite and satisfies $\Pi_P^2=\Pi_P$, $\norm{\Pi_P}=1$ and $\trace(\Pi_P)=m$. In principle, $f|_{A_k}$ need not be a constant function for $k=1,\ldots,m$. 

% We could also project, for example, on a linear function within the node as in the local linear forest \citep{friedberg2020local}.

Now fix some integer dimensions $0< k_1<\cdots<k_G=n$ and
suppose that the sequence $P_{k}=\{A_1^{(k)},\ldots,A_{k}^{(k)}\}$, $k\in\{k_1,\ldots,k_G\}$, of partitions is refining in the sense that each $A_i^{(k_{g+1})}$ is contained in some $A_j^{(k_g)}$, where $g=0,\ldots,G-1$. Write $V_{{k_g}}=V_{P_{k_g}}$ and $\Pi_{{k_{g}}}=\Pi_{P_{k_{g}}}$ for short and set $k_0=0$, $V_{k_0}=\{0\}$, $\Pi_{k_0}=0$. Then we have the inclusion $V_{{k_g}}\subset V_{{k_{g+1}}}$ and thus $\Pi_{{k_g}}\Pi_{{k_{g+1}}}=\Pi_{{k_{g+1}}}\Pi_{{k_{g}}}=\Pi_{{k_g}}$ for the projections. Moreover, $\Pi_{{k_g}}\preceq\Pi_{{k_{g+1}}}$ holds, in the sense that $\Pi_{{k_{g+1}}}-\Pi_{{k_g}}$ is positive semi-definite, because
\[ \scapro{(\Pi_{{k_{g+1}}}-\Pi_{{k_g}})f}{f}_n=\scapro{(\Pi_{{k_{g+1}}}-\Pi_{{k_g}})^2f}{f}_n =\norm{(\Pi_{{k_{g+1}}}-\Pi_{{k_g}})f}_n^2\ge0.
\]

%Motivation:
\begin{comment}
We introduce two early stopping variants for regression trees, both based on the orthogonal projection. These methods halt the tree’s growth at a depth determined by the stopping rule, without requiring information about subsequent depths. Our theory offers flexibility by allowing any node to be chosen for the next split, refining the sequence $P_k$.  This flexibility enables the choice of the next splitting node by different priority rules, leading to a unified framework for early stopping techniques. In the subsequent chapter, we will generalise these stopping techniques even further by introducing a continuously defined projection flow.
\end{comment}

%Notation
\begin{definition}\label{def:orth}
For the orthogonal projections $\Pi_{k_g}$, $g=0,\ldots,G$, consider the corresponding estimators and residuals (or in-sample training errors)
\[ \hat F_{k_g}=\Pi_{k_g} Y,\quad R_{k_g}^2=\norm{Y-\hat F_{k_g}}_n^2=\norm{(\Id-\Pi_{k_g})Y}_n^2.\]
Our non-interpolated stopping rule is given by 
\[ \hat g = \inf \{g \in \{0, \ldots, G\} \,|\, R_{k_g}^2 \leq \kappa \}\]
for some threshold value $\kappa>0$, and $\hat F_{k_{\hat g}}$ is the {\it early stopping estimator}.
\end{definition}

% [\MR only true for global and when all nodes are split from $g$ to $g+1$; otherwise the sum extends only over the nodes that are split from $k_g$ to $k_{g+1}$!]
%For the global early stopping, under the assumption that all nodes are split from $g$ to $g+1$, the residual difference between two consecutive partitions is the weighted sum over the impurity gain of each partition (see Equation \ref{eq:opt_split})
%\begin{align}
%    \norm{(\Pi_{{k_{g+1}}}-\Pi_{{k_g}})Y}_n^2 = R_{k_g}^2 - R_{k_{g+1}}^2 = \sum_{A \in P_{k_{g}}} \frac{n_{A}}{n} \mathbb{IG}(A),
%\end{align}
%where $A \in P_{k_{g}}$ denotes the partitions at iteration $g$. The first equality follows since the cross term is $\scapro{(Y-\hat{F}_{k_{g}}}{(Y - \hat{F}_{k_{g+1}} )}_n = R_{k_{g+1}}^2$ by orthogonality of the projection and the second equality holds by the recursion of the regression tree as in \citet{klusowski2023large}.

The global stopping splits all nodes $A \in P_{k_g}$ at a given generation $g$, whereas the semi-global stopping splits the best node in terms of the largest impurity gain across the partition nodes. Both early stopping methods are described in more detail in Algorithms \ref{alg:bfs_semi} and \ref{alg:bfs_global}.

\begin{algorithm}[thp]
\spacingset{1.2}
\small
\caption{Semi-global early stopping algorithm}
\label{alg:bfs_semi}
\begin{algorithmic}[1]
\REQUIRE Training sample $\{(X_i,Y_i)\}_{i=1}^{n}$, with root partition $\mathcal{P}_1=\R^d$
\FOR{$g=1,\ldots,n$}
\IF{stopping condition $R_{g}^2 \leq \kappa$ holds}
        \STATE Set stopping iteration $\hat g_{\text{semi}} \gets g$
        \STATE \textbf{Stop}
\ENDIF
\STATE Initialize next partition $\mathcal{P}_{g+1} \leftarrow \mathcal{P}_{g}$
\STATE Determine the next splitting node as
$A^{(g)} = \underset{A \in P_{g}  }{\operatorname{arg max}}\ \mathbb{IG}(A)$
\STATE Split best node $A^{(g)}$ according to the CART criteria into $A^{(g+1)}_{L}, A^{(g+1)}_{R}$ %, where $i_L \neq i_R$ and $i_L,i_R \leq g+1$
\STATE Refine partition $\mathcal{P}_{g+1} \leftarrow (\mathcal{P}_{g+1} \cup \{ A^{(g+1)}_{L}, A^{(g+1)}_{R} \}  ) \setminus \{ A^{(g)} \} $ and estimator $\hat{F}_{g+1}$
\STATE Update estimator $\hat{F}_{g+1} \gets \Pi_{g+1} Y$, where projection $\Pi_{g+1}$ is based on partition $\mathcal{P}_{g+1}$
\ENDFOR
\RETURN Semi-global early stopped regression tree estimate $\hat{F}_{\hat g_{\text{semi}} }$
\end{algorithmic}
\end{algorithm}

\subsection{Generalized projection flows}
\label{sec:gen_flow}

We generalize sequentially refining projections to a continuously parametrized projection-type family. The continuous parameter allows to balance between over- and underfitting more granularly for regression trees and avoids additional discretization errors in the analysis. This generalization is especially helpful for the global stopping rule where the partition size $k_g$ grows geometrically with the generation $g$. Thus, the estimators' performance from generations $g$ to $g+1$ might drastically change from underfitting to overfitting.  This phenomenon also occurs for early stopping in the conjugate gradient algorithm \citep{hucker2024early} and is referred to as \textit{overshooting}. 

% \RM{\textcolor{blue}{Our continuous formulation is conceptually analogous to the use of gradient flow in neural networks \citep{soudry2018implicit, jacot2018neural}.}} 

\begin{definition}
A family of selfadjoint, positive semi-definite operators $\Pi_t:L_n^2\to L_n^2$, $t\in[0,n]$, is called a {\em generalized projection flow} or {\em regularization flow} if for all $0\le s\le t\le n$
\[ \Pi_s\Pi_t=\Pi_t\Pi_s,\; \Pi_s\preceq\Pi_t,\; \norm{\Pi_t}\le 1,\; \trace(\Pi_t)=t.\]
Here $\Pi_s\preceq\Pi_t$ means that $\Pi_t-\Pi_s$ is positive semi-definite.
\end{definition}

Note that $\Pi_0$ is necessarily zero and $\Pi_n$ is the identity on $L^2_n$ due to $\trace(\Pi_0)=0$, $\trace(\Pi_n)=\dim(L_n^2)$ and $\norm{\Pi_n}\le 1$. Moreover, $\trace(\Pi_t-\Pi_s)\to 0$ as $t\downarrow s$, together with $\Pi_t-\Pi_s\succeq 0$, implies that $t\mapsto \Pi_t$ is continuous in nuclear and thus in operator norm. To fully reflect the original CART algorithm, we explicitly allow $(\Pi_t)$ to depend on the data. This concept unifies the interpolated regression tree methods of this paper.

\begin{algorithm}[thp]
\spacingset{1.2}
\small
\caption{Global early stopping algorithm}
\label{alg:bfs_global}
\begin{algorithmic}[1]
\REQUIRE Training sample $\{(X_i,Y_i)\}_{i=1}^{n}$, with root partition $\mathcal{P}_1 = \R^d$
\FOR{$g=1,\ldots,G$}
\IF{stopping condition $R_{k_{g}}^2 \leq \kappa$ holds}
        \STATE Set stopping generation $\hat{g}_{\text{glob}} \gets g$
        \STATE \textbf{Stop}
\ENDIF
\STATE Initialize next partition $\mathcal{P}_{k_{g+1}} \gets \mathcal{P}_{k_g}$
\FOR{$j=1,\ldots,k_g$, where $|\mathcal{P}_{k_g}| = k_g$}
\IF{Cardinality $|A^{(g)}_j| = 1$}
        \STATE Do not split $A^{(g)}_j$
\ELSE
\STATE Split node $A^{(g)}_j$ using CART criteria into $A^{(g+1)}_{j_L}$ and $A^{(g+1)}_{j_R}$ where $j_L \neq j_R$
\STATE Refine partition $\mathcal{P}_{k_{g+1}} \leftarrow (\mathcal{P}_{k_{g+1}} \cup \{ A^{(g+1)}_{j_L}, A^{(g+1)}_{j_R} \}  ) \setminus \{ A^{(g)}_j \}$
\STATE Update estimator $\hat{F}_{k_{g+1}} \gets \Pi_{k_{g+1}} Y$, where projection $\Pi_{k_{g+1}}$ is based on partition $\mathcal{P}_{k_{g+1}}$
\ENDIF
\ENDFOR
\ENDFOR
\RETURN Global early stopped regression tree estimate $\hat{F}_{\hat g_{\text{glob}}}$
\end{algorithmic}
\end{algorithm}

\begin{example}\label{Ex1}
For $k_g\in\{0,1,\ldots,n\}$ and generations $g=0,\ldots,G$ let orthogonal projections $\Pi_{k_g}$ on nested $k_g$-dimensional subspaces be given. Assume $k_0=0$ and $k_G=n$.
For $t\in [0,n)$ we define the linearly interpolated projections
\begin{equation*}
\Pi_t=(1-\alpha)\Pi_{k_{g(t)}}+\alpha\Pi_{k_{g(t)+1}}, \text{ where } g(t)=\max\{g\,|\,k_g\le t\},\,\alpha=\tfrac{t-k_{g(t)}}{k_{g(t)+1}-k_{g(t)}}\in[0,1).
\end{equation*}
Then $(\Pi_t)_{t\in[0,n]}$ forms a generalized projection flow.

For semi-global early stopping we have $G=n$ and $k_g=g\in\{0,\ldots,n\}$. Then $\Pi_t=(1-\alpha)\Pi_{\floor{t}}+\alpha\Pi_{\floor{t}+1}$ holds with $\alpha=t-\floor{t}$, $t\in[0,n]$.
For global early stopping, we have $k_0=0$, $k_1=1$ and then every generation at most doubles the dimension of the projection space so that $k_g<k_{g+1}\le 2k_g$ holds for $g\ge 1$. A priori, nothing more is known about the generalized projection flow $(\Pi_t)_{t\in[0,n]}$ obtained in this case. From a computational perspective, the interpolation comes at no additional cost since the required quantities are calculated in either case.

\end{example}

Although we focus on regression tree methods, the early stopping theory developed here applies in much more generality.
Let us discuss some prominent regularization methods in the scope of our framework.

\begin{example}\label{Ex2}
Consider the high-dimensional linear model where $f(x)=\scapro{\beta}{x}$ for an unknown parameter $\beta\in\R^d$. Then the ridge estimator of $F=(f(X_i))_{i=1,\ldots,n}\in\R^n$ is given by $\hat F_\lambda=S_\lambda Y$ with the smoother matrix $S_\lambda=X(X^\top X+\lambda\Id)^{-1}X^\top$ in terms of the design matrix $X=(X_1,\ldots,X_n)^\top\in\R^{n\times d}$ and the ridge parameter $\lambda\ge 0$. Its effective dimension (or effective degrees of freedom) is given by $t(\lambda)=\trace(S_\lambda)$. Setting $\Pi_t=S_{\lambda(t)}$ for the inverse function $\lambda(t)$ of $t(\lambda)$ and $\Pi_0=S_\infty=0$ we obtain a regularization flow. Formally, we need $\rank(X)=n$ to ensure $\trace(S_0)=n$, but otherwise, we can work on the lower-dimensional range of $X$ instead of the full $\R^n$.

This approach also applies to other penalized least squares approaches; see Section 5.4.1 in \citet{hastie2017elements} for a general discussion and an application to smoothing splines. The aim of early stopping in this context is to choose the penalization parameter $\lambda$ sequentially, starting from large values of $\lambda$ and decreasing it gradually. Let us stress at this point that the intrinsic parametrization of $(\Pi_t)$ by $\trace(\Pi_t)$ is only for mathematical convenience.
\end{example}

\begin{example}\label{Ex3}
Gradient descent methods fit the paradigm of early stopping even better in view of their implicit regularization when stopping the iterations early. The standard gradient descent flow (sometimes called Showalter method) for the linear model is obtained by setting $\hat \beta_0=0$, $\frac{d}{ds}\hat\beta_s=-\nabla_\beta(\frac12\norm{Y-X\beta_s}^2)$ for $s\ge 0$. This yields $\hat F_s=X\hat\beta_s=G_s Y$ with $G_s=X(X^\top X)^{-1}(\Id-e^{-sX^\top X})X^\top$. Introducing the effective dimension parametrization $t(s)=\trace(G_s)=\trace(\Id-e^{-sX^\top X})$, we obtain the regularization flow $\Pi_t=G_{s(t)}$ in terms of the inverse parametrization $s(t)$. Similarly, linear interpolation of discrete gradient descent steps gives rise to a regularization flow, compare to the Landweber method in \citet{blanchard2018optimal}.
Often, $(\Pi_t)$ is also data-dependent, resulting in nonlinear methods. 
% Entfernt:
% We refer to \citet{stankewitz2024early} for $L^2$-boosting and to \citet{hucker2024early} for conjugate gradient flows.
\end{example}

In the spirit of the data-dependent histogram regression estimators in \citet{nobel1996histogram}, any generalized projection flow defines a flow of estimators to which we can apply the early stopping methodology.

\begin{definition}
For a generalized projection flow $(\Pi_t)$, define the corresponding estimators and global  residuals (or in-sample training errors)
\[ \hat F_t=\Pi_tY,\quad R_t^2=\norm{Y-\hat F_t}_n^2=\norm{(\Id-\Pi_t)Y}_n^2.\]
\end{definition}

\begin{lemma}\label{LemRiskBound}
The squared loss of $\hat F_t$ decomposes in an approximation error, a stochastic error, and a cross term as
\[ \norm{\hat F_t-f}_n^2=\norm{(\Id-\Pi_t)f}_n^2+\norm{\Pi_t\eps}_n^2-2\scapro{\Pi_t\eps}{(\Id-\Pi_t)f}_n.\]
It satisfies the bound
\[ \norm{\hat F_t-f}_n^2\le \Big(\norm{(\Id-\Pi_t)f}_n+\norm{\Pi_t\eps}_n\Big)^2\le 2\norm{(\Id-\Pi_t)f}_n^2+2\norm{\Pi_t\eps}_n^2.\]

The approximation error $\norm{(\Id-\Pi_t)f}_n^2$ decreases continuously from $\norm{f}_n^2$ at $t=0$ to zero at $t=n$. The stochastic error $\norm{\Pi_t\eps}_n^2$ increases continuously from zero at $t=0$ to $\norm{\eps}_n^2$ at $t=n$.
\end{lemma}

% \begin{proof}
% See Appendix \ref*{proof:LemRiskBound}.
% \end{proof}

\begin{remark}
For $\Pi_t=(1-\alpha)\Pi_{k_{g(t)}}+\alpha \Pi_{k_{g(t)+1}}$  as in Example \ref{Ex1} the cross term satisfies
\[ \scapro{\Pi_t\eps}{(\Id-\Pi_t)f}_n=\alpha(1-\alpha)\scapro{(\Pi_{k_{g(t)+1}}-\Pi_{k_{g(t)}})\eps}{f}_n,\]
which follows by simple algebra, see Equation \ref{EqIntAlgebra} in the Appendix. In this case, it thus represents an interpolation error between projections, which for specific choices of $f$ and $\eps$ might be relatively large (even $\norm{\hat F_t-f}_n^2=0$ is possible for particular choices of $f,\eps\not=0$). For deterministic (or independent) $\Pi_t$, however, its expectation is zero. Its standard deviation $\frac{\alpha(1-\alpha)\sigma}{\sqrt n}\norm{(\Pi_{k_{g(t)+1}}-\Pi_{k_{g(t)}})f}_n$ is usually small, at least in the semi-global setting with $k_{g(t)+1}-k_{g(t)}=1$.
\end{remark}

The following result is fundamental to base the early stopping criterion on the residual.

\begin{lemma}\label{LemRt2}
The residuals $R_t^2$ are continuous and non-increasing in $t$ from $R_0^2=\norm{Y}_n^2$ to $R_n^2=0$.
For  $0\le s\le t\le n$ we have
\[ \norm{\hat F_t-\hat F_s}_n^2\le R_s^2-R_t^2.\]
\end{lemma}

% \begin{proof}
% See Appendix \ref*{proof:LemRt2}
% \end{proof}

\begin{remark}
Note that for orthogonal projections $\Pi_s$, $\Pi_t$ with $\Pi_s\preceq\Pi_t$ we have $\Pi_s\Pi_t=\Pi_t\Pi_s=\Pi_s$, which implies $\norm{\hat F_t-\hat F_s}_n^2= R_s^2-R_t^2$ exactly.
\end{remark}

\section{Oracle inequalities}\label{SecOI}

\subsection{Error decomposition under early stopping}
\label{sec:error_decomp}

\begin{definition}\label{def:inter}
We introduce the random {\em  balanced oracle} (depending on $f$ and $\eps$)
\begin{equation*}
 \tau_b=\inf\{t\in[0,n]\,|\, \norm{(\Id-\Pi_t)f}_n\le\norm{\Pi_t\eps}_n\}
\end{equation*}
and for  $\kappa\ge 0$ the data-dependent {\em early stopping rule}
\begin{equation*}
 \tau=\inf\{t\in[0,n]\,|\, R_t^2\le\kappa\}.
\end{equation*}
\end{definition}

By continuity and monotonicity of approximation and stochastic error, we see that $\tau_b$ exists in $[0,n]$ and the errors are balanced in the sense that $\norm{(\Id-\Pi_{\tau_b})f}_n=\norm{\Pi_{\tau_b}\eps}_n$. Equally, $\tau\in[0,n]$ exists and $R_\tau^2=\kappa$ holds, provided $\kappa\le R_0^2=\norm{Y}_n^2$. Note that the terminology 'balanced' \textit{does not} refer to a balanced tree structure, but rather to the concept of balancing the approximation and stochastic error.

We provide the main $\omega$-wise error decomposition for early stopping when compared to the balanced oracle as a benchmark.

\begin{proposition}\label{PropDistanceOracle}
The distance between the early stopping estimator and the balanced oracle estimator is bounded as
\[ \norm{\hat F_\tau-\hat F_{\tau_b}}_n^2\le \underbrace{\abs{\kappa-\norm{\eps}_n^2}}_{\text{early stopping error}}+\underbrace{2\scapro{(\Pi_{\tau_b}-\Pi_{\tau_b}^2)\eps}{\eps}_n}_{\text{interpolation error}}
+\underbrace{2\abs{\scapro{(\Id-\Pi_{\tau_b})^2f}{\eps}_n}}_{\text{cross term}}.\]
\end{proposition}

\begin{remark}\label{RemErrDec}
Ideally, the threshold $\kappa$ in early stopping should equal the squared empirical norm of the noise $\eps$. This is not accessible for the statistician and yields the main general error for the estimator obtained by early stopping. If we are to choose a deterministic threshold value $\kappa$ in early stopping, then the natural value is $\kappa=\E[\norm{\eps}_n^2]=\sigma^2$. More generally, we have
\[ \E[\abs{\kappa-\norm{\eps}_n^2}]\le \E[\abs{\kappa-\sigma^2}]+\E[\abs{\sigma^2-\norm{\eps}_n^2}^2]^{1/2}\le \E[\abs{\kappa-\sigma^2}]+n^{-1/2}\Var(\eps_1^2)^{1/2}
\]
and for $\kappa=\sigma^2$ the early stopping error will still be of order $n^{-1/2}$, which represents an intrinsic information loss due to early stopping, compare the lower bound in \citet{blanchard2018early}.

The interpolation error is due to the fact that $\Pi_{\tau_b}$ is usually not a projection, in which case it vanishes. It must be analyzed for the concrete flow $(\Pi_t)$ under consideration, which is done for the semi-global and global regression tree algorithms in Lemma \ref{LemIntPol} below. The cross-term is challenging to control because both arguments of the scalar product are random with a complex dependency structure. Section \ref{SecCT} below is devoted to controlling this term.
\end{remark}
% \begin{proof}
%     The proof is given in Appendix \ref*{app:PropDistanceOracle}.
% \end{proof}

\subsection{General oracle inequalities}
\label{sec:oracle_ineq}

We directly obtain a first oracle-type inequality in expectation under subgaussian noise. While the results could also be achieved with high probability, doing so would lead to more involved formulations.

\begin{theorem}\label{ThmOI1}
If the noise vector $\eps$ is $\bar\sigma$-subgaussian in the sense that $\E[\exp(\scapro{\lambda}{\eps}_n)]\le \exp(\norm{\lambda}_n^2\bar\sigma^2/(2n))$ for all $\lambda\in\R^n$, then the interpolation error satisfies
\begin{align}
2\E\big[\scapro{(\Pi_{\tau_b}-\Pi_{\tau_b}^2)\eps}{\eps}_n\big]& \le 16\bar\sigma^2\log(n)n^{-1}\E[\trace(\Pi_{\tau_b}-\Pi_{\tau_b}^2)].\label{EqOI1a}
\end{align}
For the risk we
obtain  the early stopping oracle-type inequality
\begin{align}
\E[\norm{\hat F_\tau-f}_n^2]
&\le 9\E\Big[\inf_{t\in[0,n]}\Big(\norm{(\Id-\Pi_{t})f}_n^2+\norm{\Pi_{t}\eps}_n^2\Big)\Big]+2\E[\abs{\kappa-\norm{\eps}_n^2}]
\nonumber\\
&\quad + 32\bar\sigma^2\frac{\log(2n)}{n}\E[\trace(\Pi_{\tau_b}-\Pi_{\tau_b}^2)] +4\E\big[\scapro{\tfrac{(\Id-\Pi_{\tau_b})^2f}{\norm{(\Id-\Pi_{\tau_b})f}_n}}{\eps}_n^2\big].\label{EqOI1b}
\end{align}
Here and in the sequel, we apply the convention
\begin{equation}\label{EqOI1Conv}
\tfrac{(\Id-\Pi_{\tau_b})^2f}{\norm{(\Id-\Pi_{\tau_b})f}_n}:=0\text{ in case }(\Id-\Pi_{\tau_b})f=0.
\end{equation}
\end{theorem}

\begin{remark}
The interpolation error bound \eqref{EqOI1a} depends on the underlying generalized projection flow. For semi-global early stopping, it will be negligible, but for global early stopping, it matters.
The numerical constants in \eqref{EqOI1b} are explicit but could be optimized, depending on the size of the different terms. This also applies to all bounds obtained in the sequel.
\end{remark}

% \begin{proof}
%     The proof is given in Appendix \ref*{app:ThmOI1}.
% \end{proof}

Next, we provide bounds for the interpolation error.

\begin{lemma}\label{LemIntPol}
Consider the generalized projection flow $(\Pi_t)$ from Example \ref{Ex1}  for semi-global and global early stopping. Then the interpolation error satisfies
\begin{align*}
32\bar\sigma^2\log(2n)n^{-1}\E[\trace(\Pi_{\tau_b}-\Pi_{\tau_b}^2)]\le \begin{cases} 8 \bar\sigma^2\log(2n)n^{-1},&\text{ semi-global,}\\  8\bar\sigma^2\log(2n)n^{-1}\E[\norm{\Pi_{\tau_b}}_{HS}^2\vee 1],& \text{ global.}\end{cases}
\end{align*}
\end{lemma}

\begin{remark}\label{RemHS}
Naturally, the global bound is larger than the semi-global.
The global bound is formulated in terms of the Hilbert-Schmidt norm $\norm{\Pi_{\tau_b}}_{HS}$ because for independent $\Pi_t$, the stochastic error satisfies
\[ \E[\norm{\Pi_{t}\eps}_n^2]=\sigma^2 n^{-1}\E[\norm{\Pi_t}_{HS}^2]\le  \frac{\sigma^2 t}{n},\]
using $\norm{\Pi_t}_{HS}^2\le \trace(\Pi_t)=t$. Standard entropy arguments, compare Remark 2 of \citet{klusowski2023large} and the proof of their Theorem 4.3, yield the following order for the dependent CART case and random $\tau_b$
\begin{equation}\label{EqKlTi} \E[\norm{\Pi_{\tau_b}\eps}_n^2]\lesssim \frac{\sigma^2 \E[\tau_b]\log^2(n)\log(dn)}{n},
\end{equation}
where the factor $\log^2(n)$ takes care of the possibly unbounded support of $\eps_i$.
This upper bound for the stochastic error at the balanced oracle has a larger order than the bounds of Lemma \ref{LemIntPol}, suggesting that the interpolation error is asymptotically negligible.
\end{remark}

% \begin{proof}
% The proof is given in Appendix \ref*{app:LemIntPol}
% \end{proof}

\subsection{Bounding the cross term}\label{SecCT}
% Intuition: PI independent of data
The cross-term is much harder to control. First, we derive a bound in the advantageous setting where the generalized projection flow is independent of  $(X_i,Y_i)_{i=1,\ldots,n}$, e.g. obtained from an independent sample.
This setting is similar to the centered forest initially introduced in a technical report of \citet{breiman2004consistency}, which simplifies the random forest. Further theoretical results under this assumption are developed by \citet{biau2012analysis} and \citet{klusowski2021sharp}.  An analysis of early stopping for deterministic (or independent) projections is conducted in \citet{blanchard2018early} but with respect to a deterministic oracle, which simplifies the treatment of the cross term. Here, the independent case serves mainly as a benchmark result.

\begin{proposition}\label{PropCTIndep}
If the generalized projection flow $(\Pi_t)$ is independent of the observations $(X_i,Y_i)_{i=1,\ldots,n}$ and  $\eps_i\sim N(0,\sigma^2)$ conditional on  $X_i$, then
\[ 4\E\big[\scapro{\tfrac{(\Id-\Pi_{\tau_b})^2f}{\norm{(\Id-\Pi_{\tau_b})f}_n}}{\eps}_n^2\big]\le  16 \frac{\sigma^2\log (\sqrt2n)}{n},
\]
where convention \eqref{EqOI1Conv} is in force.
\end{proposition}

% \begin{proof} See Appendix \ref*{proof:ct_ind}. \end{proof}

\begin{remark}
As the proof reveals, we can generalize the preceding result to subgaussian noise vectors $\eps$ for which $\E[(O\eps)_k(O\eps)_l\,|\,(X_i,(O\eps)_i^2)_{i=1,\ldots,n}]=0$ holds for $k\not=l$ and the orthogonal transformation $O$ diagonalizing the flow $(\Pi_t)$. If we require this property for all orthogonal transformations, however, it is not clear whether any interesting class of distributions beyond the Gaussian qualifies.
\end{remark}

Now, we allow the projections $(\Pi_t)$ to be data-driven and to depend arbitrarily on the data $(X_i,Y_i)_{i=1,\ldots,n}$. This aligns with the standard CART splitting criterion, analyzed theoretically in \citet{scornet2015consistency}, \citet{chi2022asymptotic} and \citet{klusowski2023large}.

\begin{proposition}\label{PropCTGeneral}
Suppose that  $\eps$ is $\bar\sigma$-subgaussian and that
 the data-driven generalized projection flow $(\Pi_t)$ is generated by semi-global or global early stopping. Then the cross term satisfies (noting convention \eqref{EqOI1Conv})
\[ 4\E\big[\scapro{\tfrac{(\Id-\Pi_{\tau_b})^2f}{\norm{(\Id-\Pi_{\tau_b})f}_n}}{\eps}_n^2\big]\le  16\frac{\bar\sigma^2 \E[\tau_b+1]\log(dn)} {n}.
\]
\end{proposition}

% \begin{proof} See Appendix \ref*{proof:ct_dep}. \end{proof}

The proof relies on bounding the number of all possible projections after $k_g$ splits, which is also the complexity measure underlying the results by \citet{nobel1996histogram, klusowski2023large} and Theorem 1 of \citet{scornet2015consistency}. For the original data-driven splitting in CART, as described in Section \ref{sec:cart_algo}, there are at most $d(n-1)$ different options for splitting in all $d$ coordinate directions and at all interstices between the $n$ data points. Thus, $(dn)^{k_g}$  bounds the number of possible projections, which leads to the entropy-type bound $\E[\log((dn)^{\tau_b+1})]=\E[\tau_b+1]\log(dn)$ in the proof.

Better bounds can be reached if we can restrict the number of possible partitions along the tree growth with high probability. For example, for additive regression functions in $s$ covariates with $s$ much smaller than $d$, Proposition 1 of \citet{scornet2015consistency} shows that asymptotically only splits in the $s$ relevant coordinate directions occur. If this can be ensured, then the bound $\E[\tau_b+1]\log(sn)$ in terms of the intrinsic dimension $s$ applies.
Further details on Proposition \ref{PropCTGeneral} are given in Appendix \ref{proof:ct_dep}.

\subsection{Main results for semi-global and global early stopping}
\label{sec:main}

We obtain the major theoretical results concerning early stopping for regression trees.

\begin{theorem}[Semi-global early stopping, independent splitting] \label{ThmSGind}
If the projection flow $(\Pi_t)$ for semi-global early stopping is independent of $(X_i,Y_i)_{i=1,\ldots,n}$ and  $\eps_i\sim N(0,\sigma^2)$ conditional on  $X_i$, then
the risk of semi-global early stopping satisfies the oracle inequality
\begin{align*}
\E[\norm{\hat F_\tau-f}_n^2]
&\le 9\inf_{t\in[0,n]}\E[\norm{\hat F_t-f}_n^2]+2\E[\abs{\kappa-\norm{\eps}_n^2}]
+ 48\frac{\sigma^2\log(2n)}{n}.
\end{align*}
\end{theorem}

\begin{remark}
As discussed in Remark \ref{RemErrDec}, the early stopping error $\E[\abs{\kappa-\norm{\eps}_n^2}]$ is usually of order $\sigma^2n^{-1/2}$ or larger and thus clearly dominates the last term in this bound. Typically, the oracle error will have a larger order than $n^{-1/2}$, having in mind that $n^{-2/(d+2)}$ is the best achievable minimax rate for estimating Lipschitz-continuous $f$ in dimension $d$ and noting $n^{-2/(d+2)}\ge n^{-1/2}$ already for $d\ge 2$. Then, the semi-global early stopping error under independent splitting will attain the same risk as the oracle,  up to a small numerical factor.

Usual convergence rate results for regression trees are even much slower than the nonparametric minimax rates and utilize modified splitting techniques; see, e.g., \citet{genuer2012variance}, \citet{biau2012analysis}, \citet{klusowski2021sharp} and the overview in \citet{zhang2024adaptive}.

\end{remark}

% \begin{proof}
% The proof is given in Appendix \ref*{app:ThmSGind}.
% \end{proof}

\begin{theorem}[Global early stopping, independent splitting]\label{ThmGind}
If the projection flow $(\Pi_t)$ for global early stopping is independent of $(X_i,Y_i)_{i=1,\ldots,n}$ and  $\eps_i\sim N(0,\sigma^2)$ conditional on $X_i$, then
the risk of global early stopping satisfies the oracle inequality
\begin{align*}
\E[\norm{\hat F_\tau-f}_n^2]
&\le 9\inf_{t\in[0,n]}\E[\norm{\hat F_t-f}_n^2]+2\E[\abs{\kappa-\norm{\eps}_n^2}]\\
&\quad
 + 8\frac{\sigma^2\log(2n)}{n}\big(\E[\norm{\Pi_{\tau_b}}_{HS}^2\vee 1] + 2\big).
\end{align*}
\end{theorem}

\begin{remark}\label{RemGlobalIndep}
For global early stopping, the interpolation error will clearly dominate the cross-term under independent splitting. As discussed in Remark \ref{RemHS}, compared to the error at the balanced oracle, the interpolation error is usually still negligible.

This relatively large interpolation error might be disappointing at first sight, but note that we compare with the error at the best-interpolated oracle. This term would disappear if we had considered the oracle error among the projections at the generations only. The positive twist is that interpolation of globally grown trees might significantly improve the estimators, which is also apparent in the simulation results below (Figure \ref{fig:rel_eff_low}, Figure \ref{fig:rel_eff_high} and Table \ref{tab:mse}).
\end{remark}

\begin{comment}
\begin{theorem}[Semi-global early stopping, dependent splitting]\label{ThmSGdep}
If the noise vector $\eps$ is $\bar\sigma$-subgaussian, then
the risk at semi-global early stopping with data-dependent splitting satisfies the oracle-type inequality
\begin{align*}
\E[\norm{\hat F_\tau-f}_n^2]
&\le 9\E\Big[\inf_{t\in[0,n]}\Big(\norm{(\Id-\Pi_{t})f}_n^2+\norm{\Pi_{t}\eps}_n^2\Big)\Big]+2\E[\abs{\kappa-\sigma^2}]
+\frac{2\Var(\eps_1^2)^{1/2}}{\sqrt n}\\
&\quad + \frac{\bar\sigma^2}{n}\Big(16\E[\tau_b+1]\log(dn)+\tfrac52 \log(n)\Big).
\end{align*}
\end{theorem}

\begin{proof} Collect terms. \end{proof}

\end{comment}

\begin{theorem}[Semi-global and global early stopping, dependent splitting]\label{ThmGdep}
If the noise vector $\eps$ is $\bar\sigma$-subgaussian, then
the risk at semi-global and global early stopping, respectively, satisfies the oracle-type inequality
\begin{align*}
\E[\norm{\hat F_\tau-f}_n^2]
&\le 9\E\Big[\inf_{t\in[0,n]}\Big(\norm{(\Id-\Pi_{t})f}_n^2+\norm{\Pi_{t}\eps}_n^2\Big)\Big]+2\E[\abs{\kappa-\norm{\eps}_n^2}]
\\
&\quad +24\frac{\bar\sigma^2\log((d\vee 2)n)}{n}\E[\tau_b+1].
\end{align*}
\end{theorem}

\begin{remark}
In the most interesting case of dependent splitting, the universal cross-term bound of order $\log(dn)n^{-1}\E[\tau_b]$ dominates the interpolation error so that the advantage of semi-global splitting disappears. It is still not larger than the standard stochastic error bound for regression trees, compare Remark \ref{RemHS}.  Compared to \citet{blanchard2018early, stankewitz2024early} and \citet{hucker2024early}, the cross-term bound is much larger because of the high statistical complexity of the standard CART algorithm. In all results of the literature on regression trees a $\log(dn)$-factor appears additionally to minimax rates over classical function spaces and the last remainder term in the oracle inequality will become asymptotically negligible.

\end{remark}

\begin{remark}
We do not consider the out-of-sample prediction error here. The transfer of bounds from the empirical to the population prediction error depends on the regression tree algorithms and often requires some additional assumptions. General results are obtained by \citet{wager2015adaptive} in the case where the splitting rule ensures a minimal percentage of the parent node's sample size also for the child nodes. Other results are obtained, e.g., by \citet{nobel1996histogram} and \citet{klusowski2023large}. Let us stress that our early stopping error bounds, though specialized to the original CART algorithm, hold in much wider generality, in particular they apply verbatim to more constrained tree growing algorithms, reducing the number of potential splits.
\end{remark}

\section{Implementation and application to real data}
\label{sec:implementation}
In this section, we discuss the nearest neighbour estimator for $\sigma^2$, cost-complexity post-pruning and an improved two-step procedure. Then empirical results for standard data sets are presented and compared.

\subsection{Noise level estimation}

We want to choose the threshold $\kappa$ close to the noise level $\norm{\eps}_n^2$ or its expectation $\sigma^2$.
%Estimators for the noise level have been widely studied in the nonparametric regression literature, notably by \citet{spokoiny2002variance, devroye2018nearest}, and in the parametric setting by \citet{sun2012scaled}. In one-dimensional cases, \citet{shen2020optimal} achieve the minimax rate under smoothness assumptions on the functional class of the regression function.
To this end, nearest neighbour estimators of $\sigma^2$ are simple and suitable for random design in higher dimensions. Under Lipschitz smoothness of the regression function $f$, they  can achieve the rate  $n^{-1/2}\vee n^{-2/d}$ \citep{devroye2013strong}, which is of smaller order than the minimax estimation rate $n^{-2/(2+d)}$ for $f$.
Our implementation makes use of the nearest neighbour estimator
\begin{equation}
     \hat{\sigma}^2 = \frac{1}{n} \sum_{i=1}^n Y_i^2 - \frac{1}{n} \sum^n_{i=1} Y_i Y_{nn(i)}, \label{eq:nn_est}
\end{equation}
proposed by \citet{devroye2018nearest} and building up on results of \citet{biau2015lectures}, where $X_{nn(i)}$ is the nearest neighbour of $X_{i}$ with respect to the Euclidean distance in $\mathbb{R}^d$ and $Y_{nn(i)}$ is the corresponding response.

%Bezug zu ES:
In practice, the estimator $\hat\sigma^2$ is biased upwards, which leads our stopping algorithms to terminate slightly earlier than intended. We can understand this in the case $\{X_{nn(i)} \rvert i=1,\ldots,n \} = \{X_{j} \rvert j=1,\ldots,n \}$, where
\begin{align*}
        \hat{\sigma}^2&= \frac12\Big(\frac{1}{n} \sum_{i=1}^n Y_i^2 - \frac{2}{n} \sum^n_{i=1} Y_i Y_{nn(i)}+\frac{1}{n} \sum_{i=1}^n Y_{nn(i)}^2\Big) \\
        \Rightarrow \E[\hat\sigma^2]&=\sigma^2+\frac1{2n}\sum_{i=1}^n\E[(f(X_i)-f(X_{nn(i)}))^2] \ge \sigma^2.
\end{align*}

Here, we do not pursue further improvements of $\hat\sigma^2$, which nevertheless would have the potential to generate somewhat better results for early stopping in practice.

% Short discussion about other estimators for noise level:
The noise estimate serves as an input to construct the early stopped regression tree, rather than being an alternative to it. While not of our interest, further noise level estimators could be investigated. For example, classical difference-based estimators \citep{rice1984bandwidth, gasser1986residual}, or more robust methods like kernel smoothing, could potentially improve the practical performance of our early stopping rules.

\subsection{Cost-complexity pruning}
\label{sec:pruning}

%Lepsky refernce. Maybe also mention \citet{gey2005model} shortly here!

%This finding aligns with analogous discussions in other areas of statistical estimation. For instance, in kernel regression, the choice between bottom-up and top-down bandwidth selection methods has similar implications. \citet{lepski1997optimal} propose starting with a small bandwidth and increasing it, akin to a bottom-up approach.

\noindent The state-of-the-art approach for determining the structure of a regression tree is cost-complexity pruning with weakest-link pruning (also called post-pruning), as introduced by \citet{breiman1984classification}. Oracle inequalities in a similar spirit to our work are developed by \citet{gey2005model}. 

Let $T\preceq T_n$ mean that $T$ is a subtree of the fully grown tree $T_n$, which has $n$ terminal nodes. The pruned subtree shares the same root node as the fully grown tree.
 %characterized by internal and terminal nodes
% \textcolor{blue}{$v$ a node, and $T_v$ the subtree rooted at node $v$. The pruned tree obtained by removing the subtree $ T_v $ is denoted as $ T - T_v $, and $ |T| $ represents the number of terminal nodes in $ T $.} 
By Theorem 10.9 of \citet{breiman1984classification}, for a given cost-complexity hyperparameter \( \lambda \) there exists a unique subtree that minimizes the risk function $T_{\lambda} = \argmin _{T \preceq T_n} \{R^2(T) + \lambda |T| \}$, 
where $R^2(T)$ denotes the residual norm of subtree $T$. After growing the tree to its full depth, the weakest-link pruning searches iteratively for the internal node to collapse, which increases $R^2(T)$ the least. This procedure is repeated iteratively until only the root node is left. As a result, a sequence of increasing hyperparameters is obtained, denoted by $\Lambda_{\text{prun}}$. 
%$\Lambda_{\text{prun}}=\{\lambda_1,\ldots,\lambda_{H-1},\lambda_{H}\}$, where $H$ denotes the number of candidates. 
This sequence of hyperparameters corresponds to a sequence of pruned subtrees with decreasing complexity, such that $\lambda = 0$ gives $T_n$, and a sufficiently large hyperparameter results in the root node. 

In our simulation and empirical applications, the optimal hyperparameter $\lambda_{\text{opt}} \in \Lambda_{\text{prun}}$ is selected using 5-fold cross-validation. The pruned regression tree estimator is then denoted as $\hat{F}_{\text{prun}}=\hat{F}(\lambda_{ \text{opt} } )$, where the simplified notation $\hat{F}_{\text{prun}}$ is used subsequently. The oracle is determined by selecting the pruned subtree that minimizes the error on a test set, denoted by subscript \( n' \). We use the relative efficiency as a performance measure, which quantifies the performance of the pruned estimator relative to its oracle, defined as $ \min _{\lambda \in \Lambda_{prun}}\|\hat{F}({\lambda}) - f \|_{n^{\prime}} /\|\hat{F}_{ \text{prun} }-f\|_{n^{\prime}}$.

% By setting $C_\lambda\left(T-T_t\right)=C_\lambda(T)$ for a given node t and solving for the cost-complexity parameter $\lambda$, we get
% \begin{align*}
%     \lambda = \frac{R^2(t) - R^2(T_t)}{|T_t| - 1}.
% \end{align*} \textcolor{orange}{(explicit $\lambda$ can be deleted)}

% The weakest link pruning searches for the node $t$ that minimizes $\lambda$. In other words, it searches for the internal node that collapses to the smallest impurity gain. This procedure is repeated iteratively until only the root node is left.  
%Let us denote this finite sequence of potential tuning parameters as 
% \textcolor{orange}{(kill this Folge:)}
% \begin{align*}
%     0=\lambda_1 < \lambda_2 < \ldots < \lambda_{H-1} < \lambda_{H},
% \end{align*}
%where $\lambda_1$ corresponds to a fully grown tree. 

\subsection{A two-step procedure}\label{sec:two_step}

%Introduction of 2step:
\noindent We propose a two-step hybrid procedure by combining the advantages of a first step, a top-down early stopping approach, and a second step, a bottom-up pruning approach. The procedure inherits the computational advantage of the early stopping rule in the first step, while the bottom-up nature of the second pruning step allows more flexibility in building the final regression tree by using less biased data. In simple truncated series estimation, a two-step approach combining early stopping with AIC-like model selection was analyzed in \citet{blanchard2018early}. 

%Algorithmic description
The procedure is described in Algorithm \ref{alg:2step}. We consider the global early stopping at generation $\hat g$, according to Definition \ref{def:orth} with threshold $\kappa$. To further reduce the risk of potential underfitting, we consider the tree obtained at generation $\hat g+1$ as a starting point for cost-complexity pruning. The sequence of pruning parameters $\Lambda_{\text{2step}}$ is different since we do not use the fully grown tree $T_n$, as in the previous section. Further, the number of potential hyperparameters to consider is substantially smaller after stopping early in the first step, such that we do not need to consider the full path of pruning hyperparameters as in Section \ref{sec:pruning}.
%$\lvert \Lambda_{\text{2step}} \rvert < \lvert \Lambda_{\text{prun}} \rvert$.

The optimal tuning parameter $\lambda^{'}_{\text{opt}} \in \Lambda_{\text{2step}}$ is determined by 5-fold cross-validation. Note that both hyperparameter choices $\lambda_{\text{opt}}$ (pruning) and $\lambda^{'}_{\text{opt}}$ (two-step) are not guaranteed to be equal. Our empirical results indicate, however, that the two-step estimator is often quite similar to the pruned estimator from Section \ref{sec:pruning}. Similar to the pruned estimator, we define the relative efficiency of the two-step estimator as 
$ \min _{ \lambda\in \Lambda_{2step }}\|\hat{F}(\lambda ) - f \|_{n^{\prime}}/\|\hat{F}_{ \text{2step} }-f\|_{n^{\prime}}$

A theoretical analysis of this two-step procedure, while desirable, is beyond the scope of this work and left for future research. It is particularly challenging due to the intricate dependency between the first stage (where the early stopping rule selects a subtree) and the second stage results (where cross-validation is performed to prune that specific subtree). We solely aim to introduce the procedure and demonstrate its empirical performance. 
% oracles are defined by minimizing the test set prediction error over the paths of tuning parameter $\Lambda_{\text{prun}}$ and $\Lambda_{\text{2step}}$, respectively. The two-step oracle is determined by $\min _{\lambda \in \Lambda_{\text{2step}}}\|\widehat{F}(\lambda) - f \|_{n^{\prime}}$.

\begin{algorithm}
\spacingset{1.2}
\small
\caption{Two-step procedure}
\label{alg:2step}
\begin{algorithmic}[1]

\STATE Grow the early stopped regression tree as in Algorithm \ref{alg:bfs_global}, resulting in estimator $\hat{F}_{\hat g_{\text{glob}}}$
% \STATE Stop the CART algorithm at iteration $\hat g$ by the data-driven stopping rule (\ref{def:orth}) with the tree estimator $\hat{F}_{\hat g_{\text{glob}}}$
\STATE Apply cost-complexity pruning to get the sequence of cost-complexity hyperparameters for $\hat{F}_{\hat g_{\text{glob}} +1}$, denoted as $\Lambda_{2step}$

% \begin{align*}
%     \lambda^{\hat g_{\text{glob}} + 1}_1 < \lambda_2 < \ldots < \lambda_{H-1} < \lambda_{H}
% \end{align*}
% If the stopping condition for $\hat g$ did not occur, the pruning path and the two-step path are the same such that $\Lambda_{\text{2step}} = \Lambda_{\text{prun}}$
\STATE Determine the optimal hyperparameter $\lambda^{'}_{\text{opt}} \in \Lambda_{2step }$ by 5-fold cross-validation
\STATE Output: Two-step regression tree estimator $\hat{F}_{\text{2step}} = \hat{F}_{\text{2step}}(\lambda^{'}_{\text{opt}})$
\end{algorithmic}
\end{algorithm}

\subsection{Data application}

\begin{table}[tbp]
\spacingset{1.2}
\centering
\footnotesize
\begin{tabular}{l|l|l|l|l|l|l}
\hline
Data set & $d, n$ & Pruning & Global & Global Int & Two-Step & Semi \\ \hline
Boston       & $14, 506$  & 3.89 (23/17s)  & 4.87 (8/0.3s) & 5.12 (8/0.3s) & 3.97 (8/1s) & 5.35 (5/0.6s) \\
Comm.  & $100, 1994$ & 0.15 (8/158s) & 0.15 (8/3s) & 0.16 (8/3s) & 0.15 (8/14s) & 0.17 (26/7s) \\
Abalone  & $8, 4177$ & 2.34 (20/112s) & 2.38 (16/2s) & 2.41 (16/2s) & 2.33 (22/8s) & 2.58 (51/2s) \\
Ozone        & $9, 330$ & 4.75 (8/8s) & 4.72 (8/0.1s) & 4.68 (8/0.1s) & 4.74 (10/1s) & 5.05 (7/0.2s) \\
Forest    & $11, 517$ & 38.79 (6/10s) & 32.94 (4/0.08s) & 31.98 (4/0.08s) & 31.51 (3/0.4s) & 32.81 (3/0.1s) \\
\hline
\end{tabular}
\vspace{-0.3cm}
\caption{Median RMSE (in brackets: median terminal nodes/median run times in seconds) for different datasets and methods.}
\label{tab:rmspe}
\end{table}

We apply the early stopped regression tree on the following publicly available benchmark datasets.
The datasets \textit{Ozone, Forest Fires, Abalone, Communities, and Crime} were sourced from the UCI Machine Learning Repository. In addition, the \textit{Boston} and \textit{Ozone} datasets were accessed via the \texttt{MASS} and \texttt{mlbench} packages in R, respectively. 

%Boston: predict the value of owner-occupied homes in the suburbs of Boston using census data from 1970. (Ravikumar et al., 2009;Lin and Zhang, 2006)

% Estimation setup:
We split the data randomly into $90\%$ training data and $10\%$ test data. The estimators are fitted using the training data, and their performance is evaluated on the untouched test data. This procedure is repeated $M=300$ times, and the median across the Monte Carlo iterations is reported. As a performance measure, we calculate the root mean squared error (RMSE) on the test set. The critical value for the early stopped estimator is taken to be $\kappa = \hat{\sigma}^2$, where the nearest neighbour estimator $\hat{\sigma}^2$ is obtained on the training data as outlined in \eqref{eq:nn_est}.
% Refer to the results and describe shortly:
The RMSE for different data sets is reported in Table \ref{tab:rmspe}, together with the number of terminal nodes and the computational run times.

The prediction performance of the early stopping methods and the cost-complexity pruning are seen to be on par. The running time, however, of the early stopped estimators compared to the pruning run times is faster by a factor of around 50, consistently across all data sets. The reason is that cross-validation is not required for our proposed stopping methods, and that the early stopped regression tree is not grown until its full depth. The number of terminal nodes is similar across the estimators, with a tendency for semi-global stopping to have slightly more nodes than global stopping. Even in the two-step procedure, which relies on cross-validation in the second step, the number of potential trees is drastically limited by applying early stopping in the first step.
The early stopping methods outperform pruning in the ozone and forest datasets. For the Boston data, the pruning and two-step methods perform better than the pure early stopping methods.

\section{Simulation results}
\label{sec:sim}
%Abgrenzen von RF, BART, etc:
Our goal is to introduce a computationally efficient and theoretically principled alternative to cost-complexity pruning \citep{breiman1984classification} for constructing a single, interpretable regression tree. Thus, the pruned regression tree serves as our primary benchmark for the simulation. 
While ensemble methods like Random Forests \citep{breiman2001random} often yield superior prediction accuracy, they sacrifice interpretability and lead to additional complexity in terms of hyperparameter tuning. 

% The simulations are divided into a low-dimensional setting ($d=5$) and a moderately high-dimensional setting ($d=30$) for sample size $n=1000$. Additional simulations with higher dimensions ($d=100$) and smaller sample size ($n=100$) are conducted to underline the scope of the proposed algorithms. 

%\subsection{Simulation A: low-dimensional setting}
\paragraph*{Simulation A: low-dimensional setting}
Following \citet{chaudhuri2023cross}, we adjust the rectangular, circular, and elliptical signals for random design in dimension $d=5$. Additionally, we include a trigonometric signal.
The design  is uniform on the unit hyper-cube with $ X_i\overset{\mathrm{iid}}{\sim} U(0,1)^{5} $, $i=1,\ldots,n$, and  $n = 1000$ observations for both the training and test sets. In $M=300$ Monte Carlo runs we simulate data according to \eqref{EqRegr} with i.i.d. noise variables $\varepsilon_i \sim N(0,1)$ and regression functions for $x=(x_1,\ldots,x_5)$.
The functions used are \textit{Rectangular}: $f(x) = \mathbf{1}(\tfrac{1}{3} \le x_1, x_2 \le \tfrac{2}{3})$; \textit{Circular}: $f(x) = \mathbf{1}((x_1 - \tfrac{1}{2})^2 + (x_2 - \tfrac{1}{2})^2 \le \tfrac{1}{16})$; \textit{Sine cosine}: $f(x) = \sin(x_1) + \cos(x_2)$; and \textit{Elliptical}: $f(x) = 20 \exp(-5((x_1 - \tfrac{1}{2})^2 + (x_2 - \tfrac{1}{2})^2 - 0.9(x_1 - \tfrac{1}{2})(x_2 - \tfrac{1}{2})))$.

% \begin{align*}
% \textit{Rectangular}: &\quad f(x) = {\bf 1}\Big(\tfrac{1}{3} \leq x_1, x_2 \leq \tfrac{2}{3}\Big), \\
% \textit{Circular}: &\quad f(x) = {\bf 1}\Big( (x_1 - \tfrac{1}{2})^2 + (x_2 - \tfrac{1}{2})^2 \le \tfrac{1}{16}\Big), \\
% \textit{Sine cosine}: &\quad f(x) = \sin(x_1) + \cos(x_2),\\ %  f(x) = \sin(3 \pi x_1) + \cos(5 \pi x_2). \\
% \textit{Elliptical}: &\quad f(x) = 20 \exp\Big(-5\big((x_1 - \tfrac{1}{2})^2 + (x_2 - \tfrac{1}{2})^2 - 0.9(x_1 - \tfrac{1}{2})(x_2 - \tfrac{1}{2})\big)\Big). 
% \end{align*}

%\subsection{Simulation B: high-dimensional setting}
\paragraph*{Simulation B: high-dimensional setting}
The covariates are drawn as $X_i \overset{\mathrm{iid}}{\sim} U(-2.5, 2.5)^{30}$ and the noise variables as $\varepsilon_i \sim N(0,1)$. 
We generate $n=1000$ observations each for the training and the test set, with a regression function determined by the additive model $f(x) = g_1(x_1) + g_2(x_2) + g_3(x_3) + g_4(x_4)$, where $x \in \R^{30}$.
The functions $g_j$ are taken out of the class of smooth functions, step functions, linear splines or Hills-type functions \citep{haris2022generalized}, see the illustration in Figure \ref{fig:additive_dgp} in the \hyperref[simulation]{Appendix}.

\subsection{Estimator-specific relative efficiency}

\begin{table}[tbp]
\spacingset{1.2}
\centering
\begin{tabular}{l|l|l|l}
Method & Stopping rule & Estimator & Relative efficiency  \\
\hline
Pruning & Section \ref{sec:pruning} & $\hat{F}_{\text{prun}}$ & $\min_{\lambda \in \Lambda_{\text{prun}}}\|\hat{F}(\lambda) - f\|_{n^{\prime}} / \|\hat{F}_{\text{prun}}-f\|_{n^{\prime}}$ \\
Global & Definition \ref{def:orth} & $\hat{F}_{\hat{g}_{\text{glob}}}$ & $\min_{t \in [0, n]}\|\hat{F}_{t} - f\|_{n^{\prime}} / \|\hat{F}_{\hat{g}_{\text{glob}}} - f\|_{n^{\prime}}$ \\
Global Int & Definition \ref{def:inter} & $\hat{F}_{\tau}$ & $\min_{t \in [0, n]}\|\hat{F}_{t} - f\|_{n^{\prime}} / \|\hat{F}_{\tau} - f\|_{n^{\prime}}$ \\
Two-Step & Section \ref{sec:two_step} & $\hat{F}_{\text{2step}}$ & $\min_{\lambda\in \Lambda_{\text{2step}}}\|\hat{F}(\lambda) - f\|_{n^{\prime}} / \|\hat{F}_{\text{2step}}-f\|_{n^{\prime}}$ \\
Semi-global & Definition \ref{def:orth} & $\hat{F}_{\hat{g}_{\text{semi}}}$ & $\min_{g \in \{0,\ldots,n\}}\|\hat{F}_{g} - f\|_{n^{\prime}} / \|\hat{F}_{\hat{g}_{\text{semi}}} -f\|_{n^{\prime}}$ \\
\end{tabular}
\vspace{-0.3cm}
\caption{Overview of stopping methods and the estimator-specific relative efficiency. The errors are calculated on the test set, denoted by subscript $n'$.}
\label{tab:rel_eff_def}
\end{table}

We assess the adaptation of early stopping to the best (oracle) estimator obtained along the tree growing procedure by its relative efficiency for each Monte Carlo run, which is a standard technique in the evaluation of regularization methods. The relative efficiency is evaluated on the unseen test data and is, by definition, the better, the closer it comes to the best possible value 1. The respective definitions are given in Table \ref{tab:rel_eff_def}. For the non-interpolated global early stopping estimators, we take as a benchmark the interpolated global oracle estimator, which does not suffer from possible overshooting from one generation to the other, compare also Remark \ref{RemGlobalIndep}. Implementation details on the post-pruning procedure are given in Appendix \ref{app:pruning}.
Keeping in mind that the different sequential tree growing algorithms naturally give different regularization paths $(\hat F_t,t\in[0,n])$, we compare the practical performance of the corresponding oracle estimators in Appendix \ref{SecOrComp}.

% \textit{Sine Cosine}: Let $x \in [0,1]^2$ be equidistant with $n=64^2$.  We generate $Y= \sin(3 \pi x_1) + \cos(5 \pi x_2) + \varepsilon,\ \text{with}\ \varepsilon \sim \text{Laplace}(0,1)$  % Ursprünglich war es\sim N(0,1)$.

% Following \citet{chaudhuri2023cross} we replicate the rectangular, circular and smooth signal. See Section XYZ for heat maps. Let $L_{d, n}:=\{1, \ldots, n\}^d$ where $N=n^d$ and $\varepsilon \sim \text{Laplace}(0,1)$. We set $d=2$ and $n=64$. The DGPs are given as: \\
% \\
% \textit{Rectangular}: The true signal $f$ is such that, for every $\left(i_1, i_2\right) \in$ $L_{2, n}$,
% $$
% f_{\left(i_1, i_2\right)}= \begin{cases}1 & \text { if } n / 3 \leq i_1, i_2 \leq 2 n / 3 \\ 0 & \text { otherwise. }\end{cases}
% $$
% \textit{Circular}: The true signal $f$ is such that, for every $\left(i_1, i_2\right) \in L_{2, n}$,
% $$
% f_{\left(i_1, i_2\right)}= \begin{cases}1 & \text { if } \sqrt{\left(i_1-n / 2\right)^2+\left(i_2-n / 2\right)^2} \leq n / 4 \\ 0 & \text { otherwise. }\end{cases}
% $$
% \textit{Smooth}: The true signal $f$ is such that, for every $\left(i_1, i_2\right) \in L_{2, n}$, we have $f_{\left(i_1, i_2\right)}=g\left(i_1 / n, i_2 / n\right)$, where
% $$
% g(x, y)=20 \exp \left(-5\left\{(x-1 / 2)^2+(y-1 / 2)^2-0.9(x-1 / 2)(y-1 / 2)\right\}\right), \quad 0 \leq x, y \leq 1.
% $$

% Refer to the figures and the tables. Explain what they show.

\begin{figure}[tbp]
    \centering
    \begin{subfigure}[b]{0.4\textwidth}
        \centering
        \includegraphics[width=\textwidth]{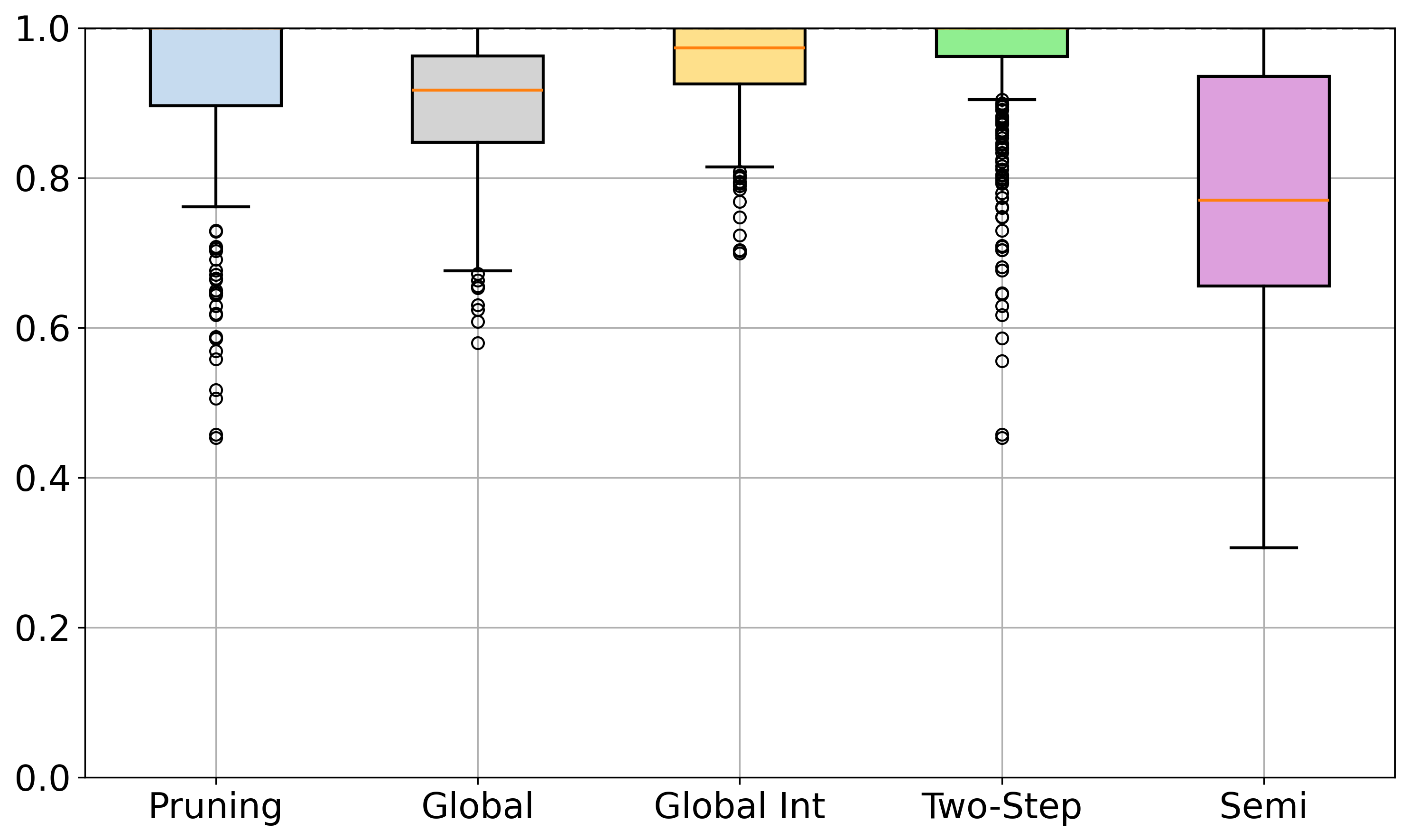}
        \caption{Rectangular function}
    \end{subfigure}
    \hspace{-0.2cm}
    \begin{subfigure}[b]{0.4\textwidth}
        \centering
        \includegraphics[width=\textwidth]{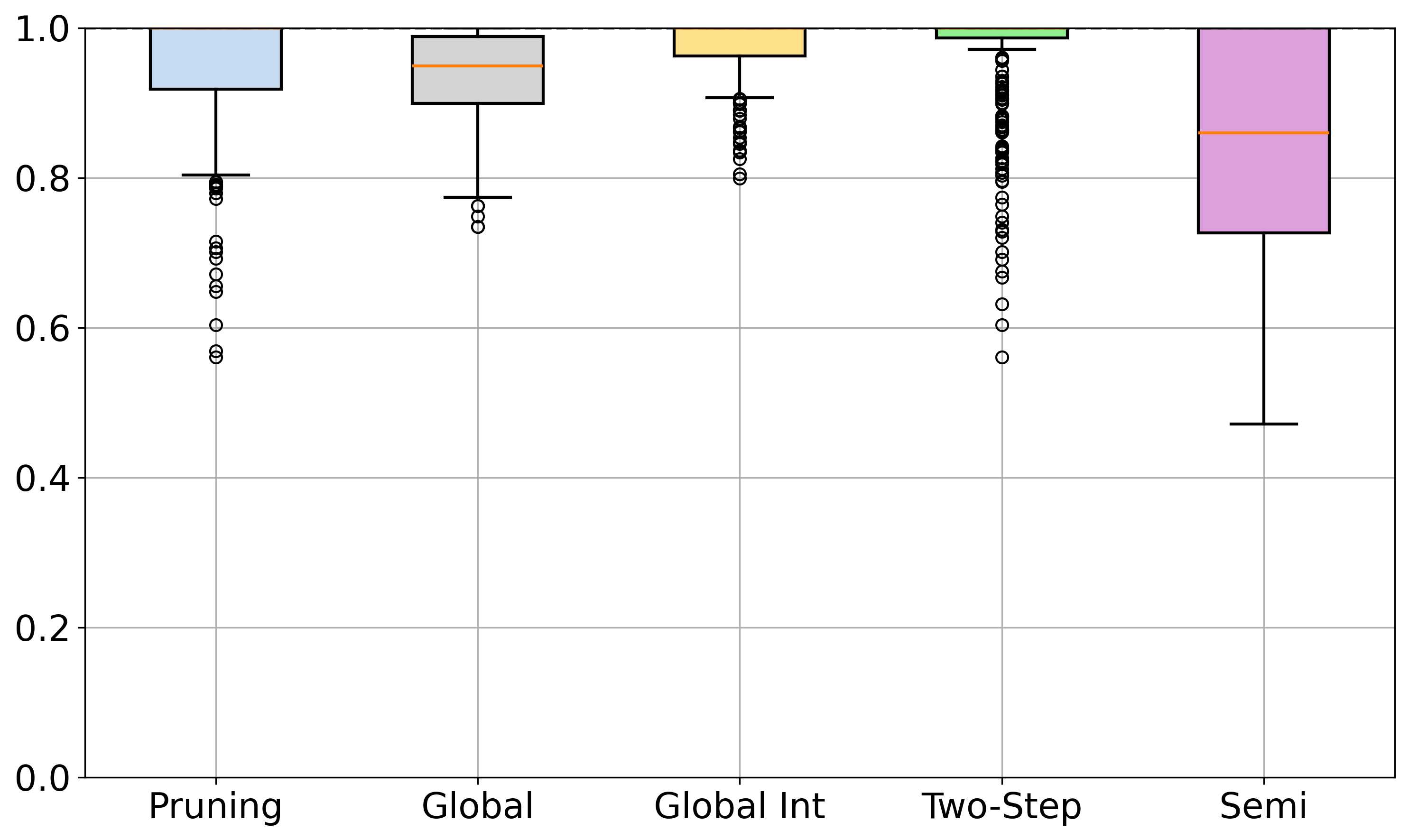}
        \caption{Circular function}
    \end{subfigure}
    \begin{subfigure}[b]{0.4\textwidth}
        \centering
        \includegraphics[width=\textwidth]{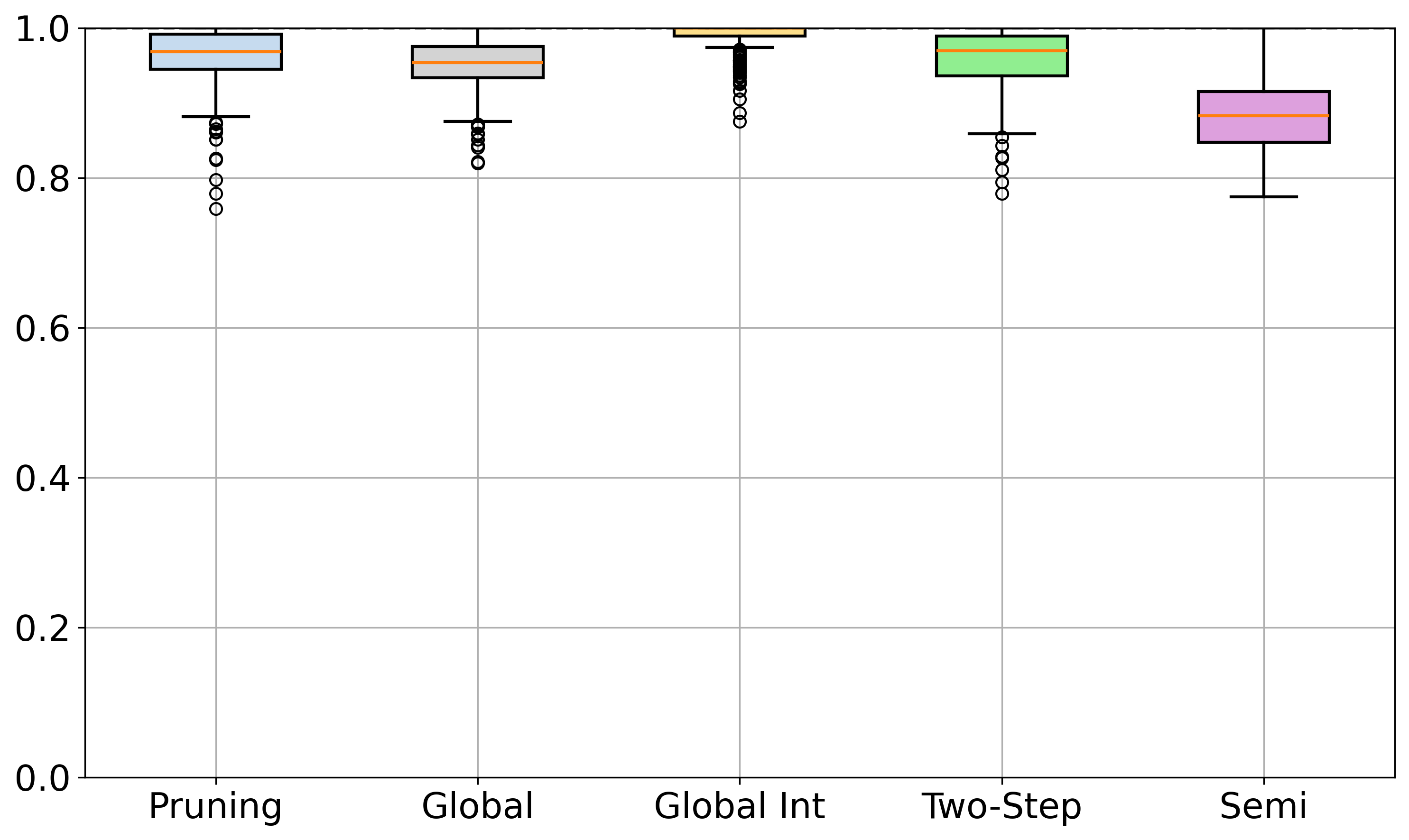}
        \caption{Sine cosine function}
    \end{subfigure}
    \hspace{-0.2cm}
    \begin{subfigure}[b]{0.4\textwidth}
        \centering
        \includegraphics[width=\textwidth]{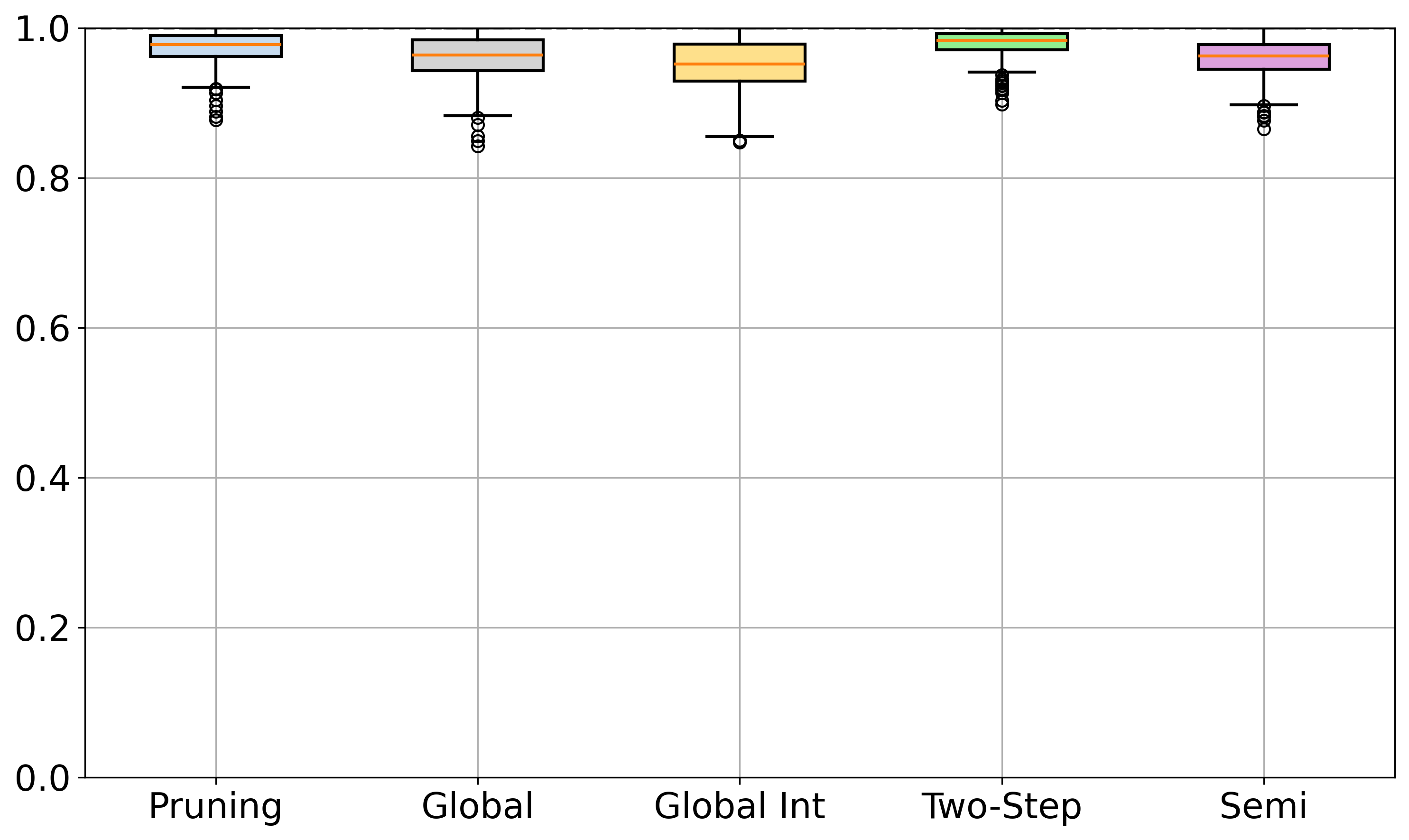}
        \caption{Elliptical function}
    \end{subfigure}
    \caption{Relative efficiency for low-dimensional signals. Higher values are better. }
    \label{fig:rel_eff_low}
\end{figure}

% We observe that the global estimator’s relative efficiency is closer to one than that of the semi-global estimator for the rectangular and circular functions (see Figure \ref{fig:rel_eff_low}). This phenomenon is directly linked to the differences in the oracle errors between the global and semi-global estimators. Specifically, the semi-global oracle error is smaller than the global oracle error for these functions, as evident from the ratio plots in Figure \ref{fig:ratios}.

% Consequently, the semi-global oracle produces a sparser tree structure, leading to a smaller variance than the global estimator. This effect is particularly prominent in functions without smooth transitions between signal and noise, such as the rectangular and circular functions.

% % Why is semi relative efficiency more variable for rectangular/circular in Figure \ref{fig:rel_eff_low} :
% Moreover, the semi-global estimator’s relative efficiency exhibits greater variability for the rectangular and circular functions in Figure \ref{fig:rel_eff_low}. This increased variability can be attributed to the sharp transitions between signal and noise in these functions. The semi-global estimator, which avoids splitting in noisy regions, may be more sensitive to variations in the data when such abrupt changes are present, leading to fluctuations in its relative efficiency.

\begin{figure}[tbp]
    \centering
    \begin{subfigure}[b]{0.4\textwidth}
        \centering
        \includegraphics[width=\textwidth]{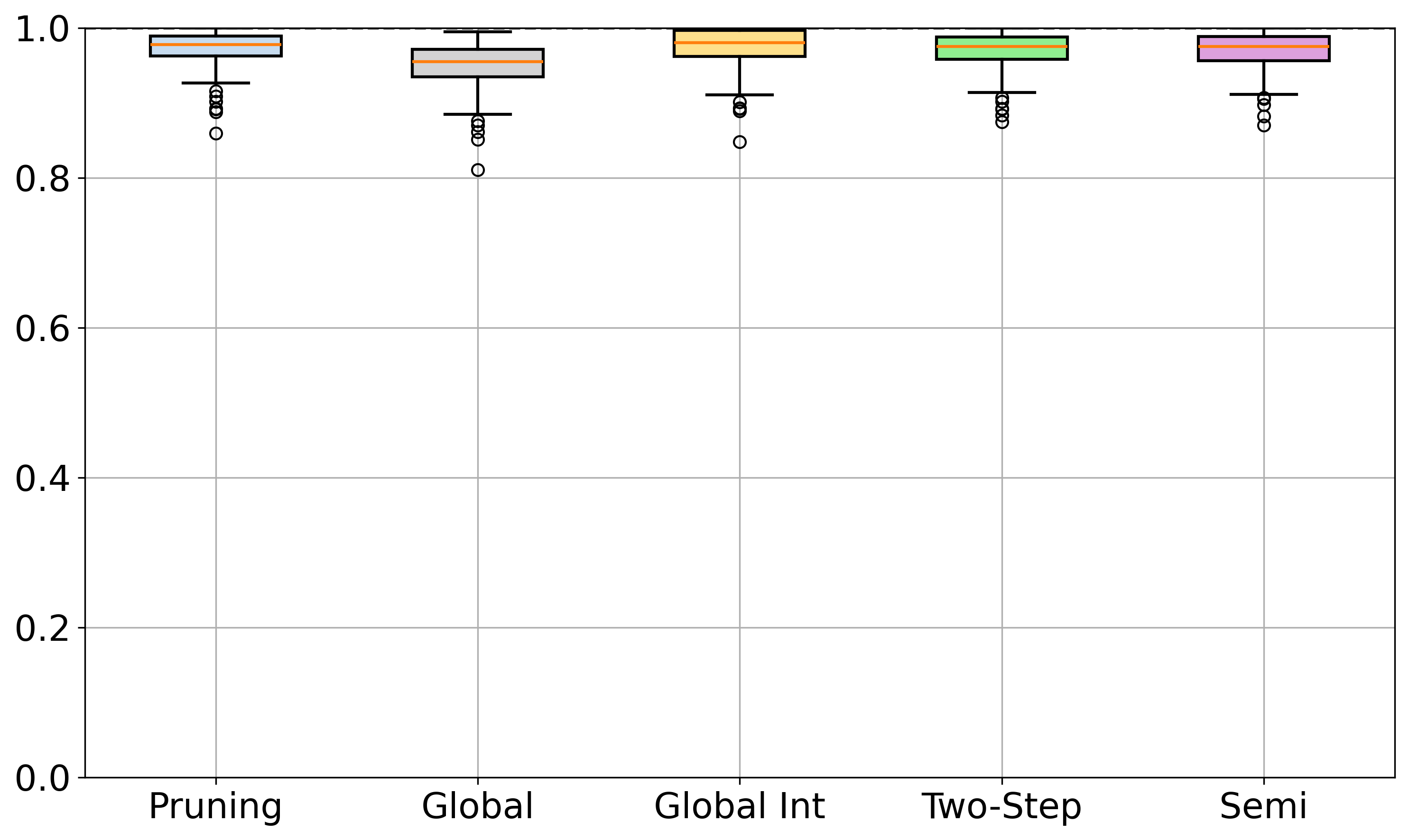}
        \caption{Smooth functions}
    \end{subfigure}
    \hspace{-0.2cm}
    \begin{subfigure}[b]{0.4\textwidth}
        \centering
        \includegraphics[width=\textwidth]{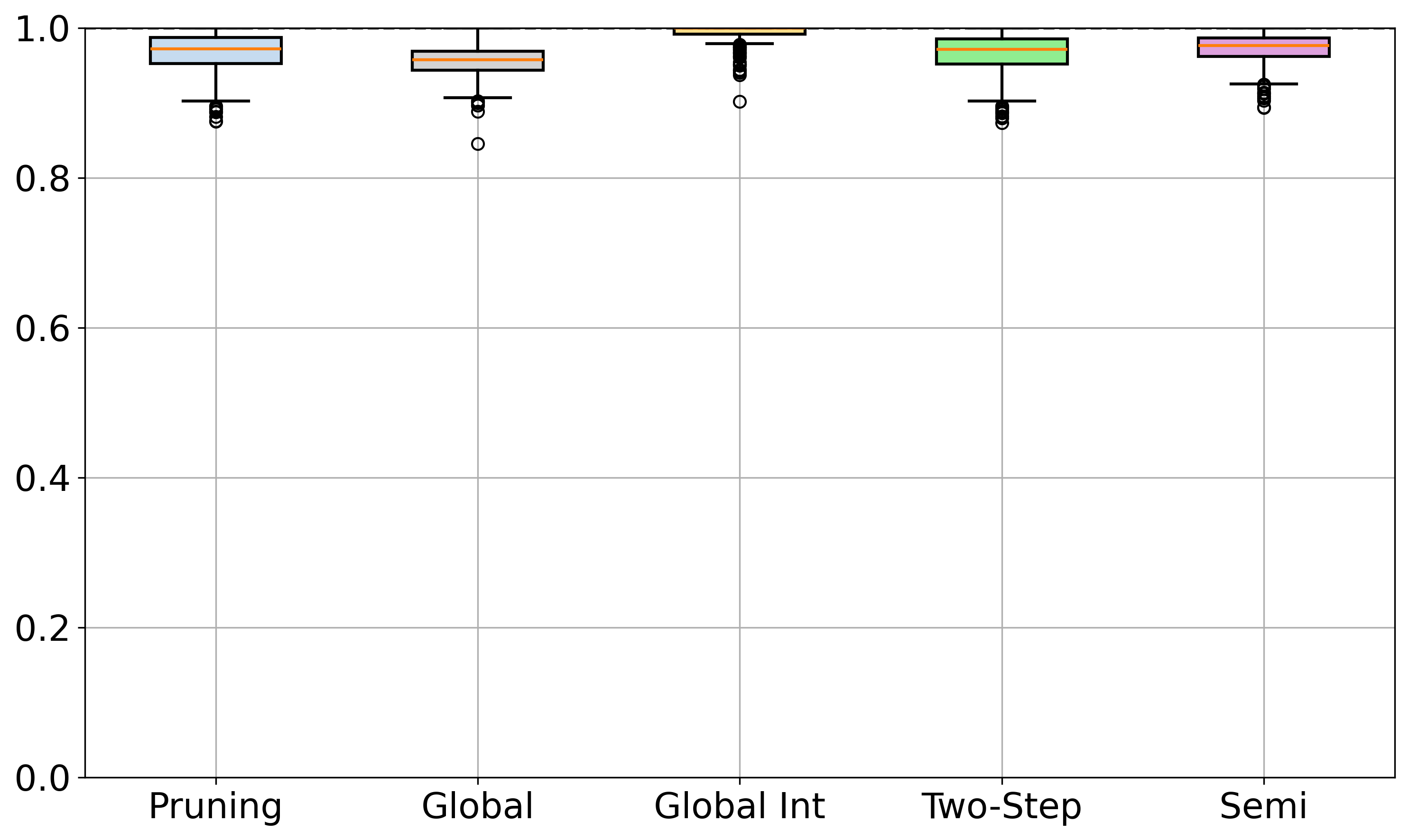}
        \caption{Step functions}
    \end{subfigure}
    \begin{subfigure}[b]{0.4\textwidth}
        \centering
        \includegraphics[width=\textwidth]{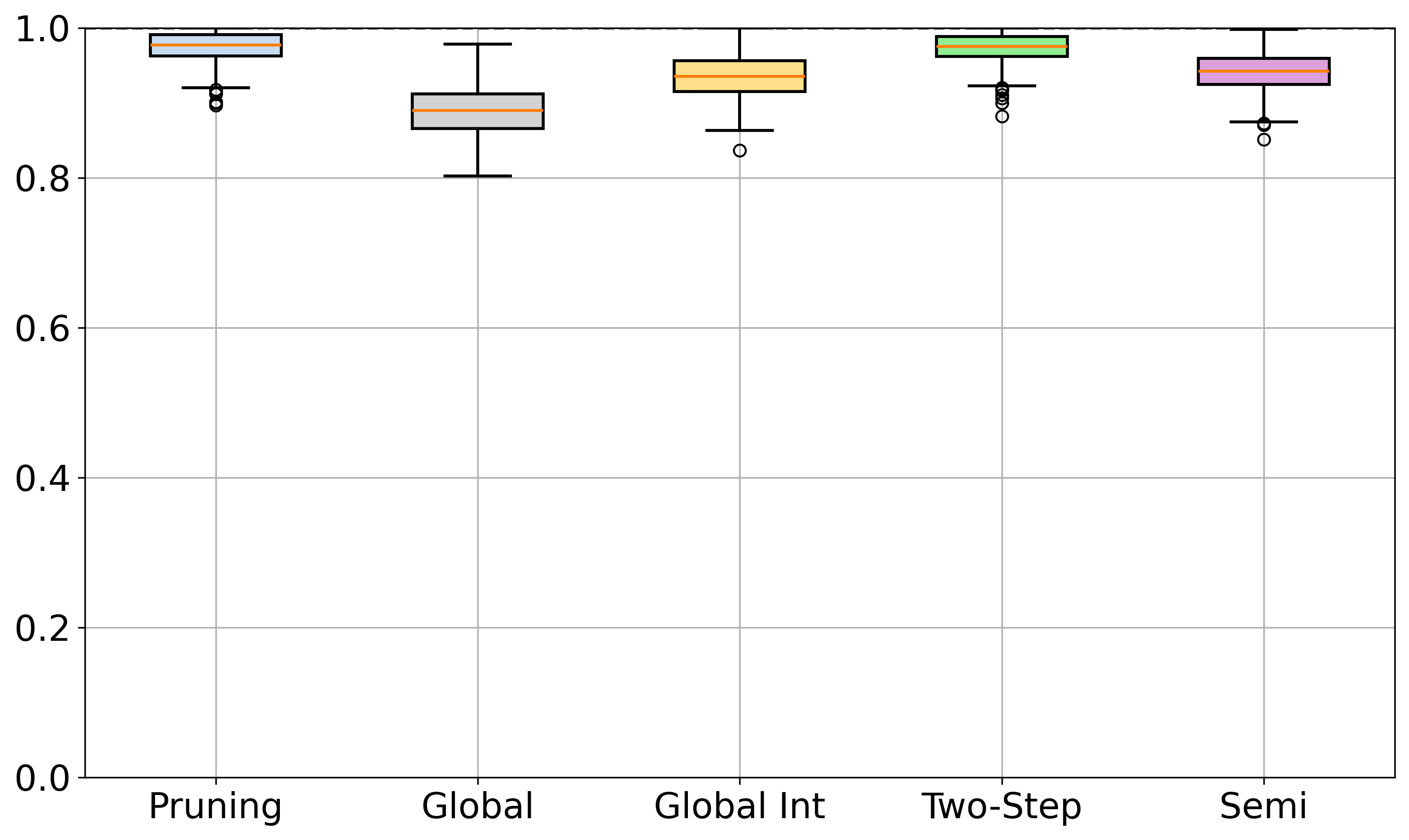}
        \caption{Piecewise linear functions}
    \end{subfigure}
    \hspace{-0.2cm}
    \begin{subfigure}[b]{0.4\textwidth}
        \centering
        \includegraphics[width=\textwidth]{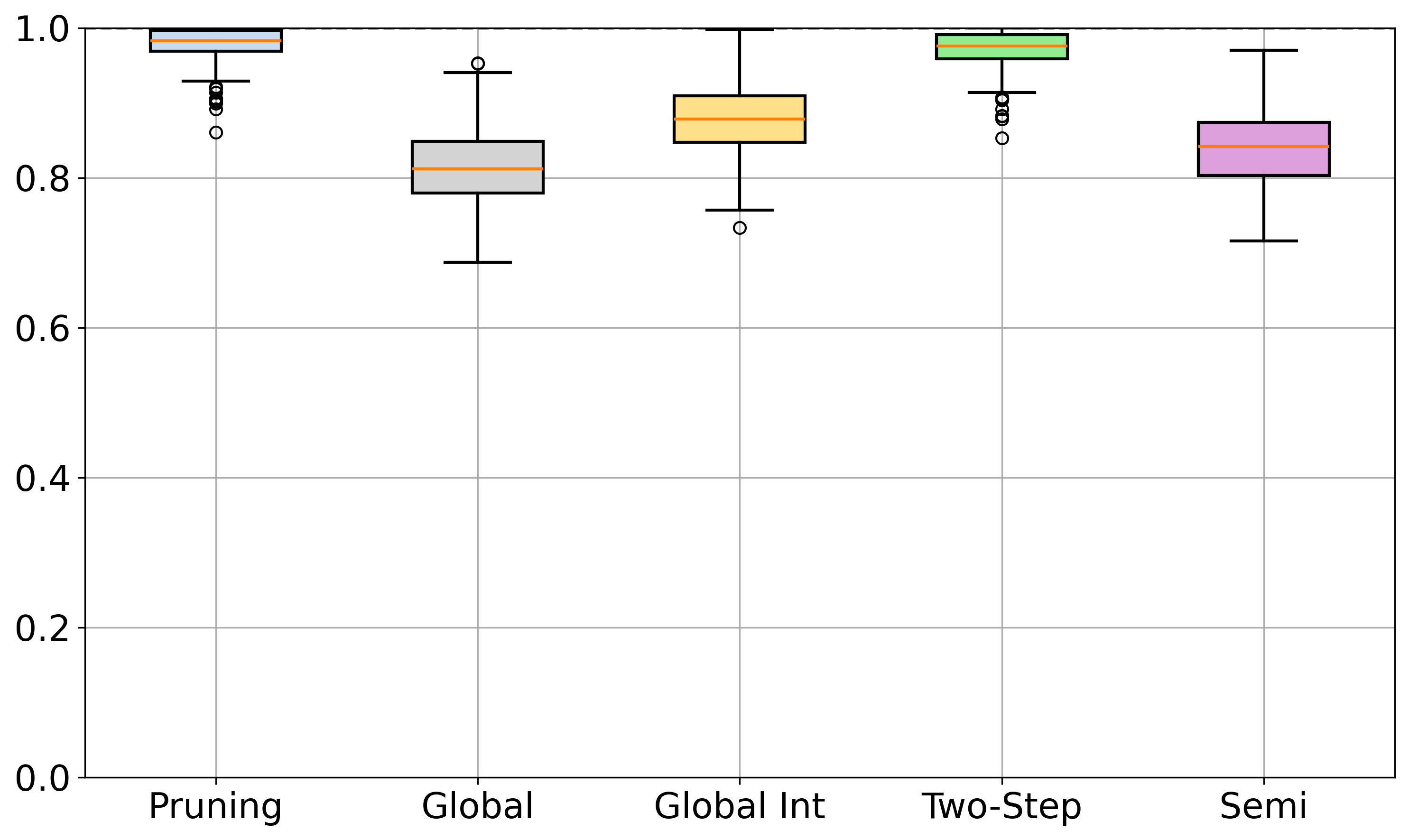}
        \caption{Hills-type functions}
    \end{subfigure}
    \caption{Relative efficiency for high-dimensional additive. Higher values are better.}
    \label{fig:rel_eff_high}
\end{figure}

The estimator-specific relative efficiency, with threshold value $\kappa = \sigma^2$, is illustrated in boxplots for the low- and high-dimensional simulations in Figures \ref{fig:rel_eff_low} and \ref{fig:rel_eff_high}, respectively. The relative efficiency is constantly above 0.5 so that the errors for the early stopping and pruning estimators are always smaller than twice the associated oracle errors. Often, the median relative efficiency is around 0.9, which is surprisingly good for a nonparametric adaptation problem. This applies quite homogeneously across all five methods and eight simulation settings.

%Effect 1:interpolation is better than no interpolation
Comparing the different results in Figures \ref{fig:rel_eff_low} and  \ref{fig:rel_eff_high}, we note three conspicuous features. First, interpolation clearly improves the performance for global early stopping. Second, the semi-global relative efficiency closely matches the global relative efficiency across most simulated functions, except the {\it rectangular} and {\it circular} functions in Figure \ref{fig:rel_eff_low}. This is evidently explained by the smaller semi-global oracle error for these examples, compare Figure \ref{fig:ratios}(a) below. The overall semi-global prediction error for these examples, as well as the other low dimensional examples, is in fact better than the global early stopping error, see Table \ref{tab:mse} below. Third, the relative efficiencies in high dimensions tend to be higher, caused by larger oracle errors, with the exception of the {\it Hills-type function} model, where the global and the semi-global early stopping suffers from stopping too late. This is indicated by the number of terminal nodes of the estimator compared to its oracle in Table \ref{tab:times_nodes}.
%Effect 3:Pruning (shortly!)
Finally, there is no significant difference between the relative efficiencies of early stopping and pruning for low-dimensional and high-dimensional simulations, except for the {\it Hills-type} example, for the aforementioned reason.

\begin{table}[tbp]
\spacingset{1.2}
\centering
\footnotesize
\setlength\tabcolsep{3pt} % Reduce space between columns
\begin{tabular}{lllllll}
\hline
 & Pruning & Global & Global Int & Two-Step & Semi & Deep \\ \hline
Rectangular & 0.21 (0.20) & 0.33 (0.30) & 0.31 (0.30) & 0.20 (0.19) & 0.30 (0.22) & 1.05 \\
Circular & 0.25 (0.24) & 0.36 (0.34) & 0.35 (0.34) & 0.24 (0.24) & 0.30 (0.24) & 1.06 \\
%Sine Cosine (old example) & 0.59 (0.57) & 0.73 (0.70) & 0.71 (0.70) & 0.60 (0.57) & 0.86 (0.76) \\
Sine cosine & 0.21 (0.19) & 0.21 (0.18) & 0.20 (0.18) & 0.20 (0.19) & 0.22 (0.19) & 1.04 \\
Elliptical & 1.12 (1.09) & 1.29 (1.23) & 1.30 (1.23) & 1.14 (1.11) & 1.21 (1.16) & 1.32 \\
Additive smooth & 1.75 (1.69) & 1.99 (1.89) & 1.93 (1.89) & 1.75 (1.69) & 2.01 (1.95) & 2.15 \\
Additive step & 1.45 (1.40) & 1.62 (1.55) & 1.55 (1.55) & 1.45 (1.40) & 1.63 (1.59) & 1.81 \\
Additive linear & 1.55 (1.52) & 1.77 (1.58) & 1.69 (1.58) & 1.56 (1.52) & 1.76 (1.65) & 1.95 \\
Additive hills & 0.85 (0.83) & 1.09 (0.88) & 1.00 (0.88) & 0.85 (0.83) & 1.14 (0.96) & 1.45 \\
\hline
\end{tabular}
\vspace{-0.3cm}
\caption{Median RMSE on the test set with $\kappa=\sigma^2$ (median oracle RMSE).}
\label{tab:mse}
\end{table}
\subsection{Total prediction error}

%Intro
The prediction error is measured as the root mean squared error (RMSE), see the denominator in Table \ref{tab:rel_eff_def}.  We present the median prediction error across the Monte Carlo runs together with the respective median oracle error in Table \ref{tab:mse} for noise level $\sigma^2$.

%Two-step:
The two-step oracle achieves a performance comparable to the pruning oracle while being faster to compute (see Table \ref{tab:times_nodes}).
Hence, the two-step procedure effectively combines the computational efficiency of the global early stopping method with the bottom-up nature of cost-complexity pruning.
%Deep tree:
For comparison, growing the tree to its full depth yields higher median RMSE values than all the competing algorithms. Thus, we effectively improve the prediction performance by stopping early.

% Run times table & number of terminal nodes
The computational run times are reported in Table \ref{tab:times_nodes}, and it shows that the early stopped regression trees are faster than the pruning by a large magnitude. Further, the number of terminal nodes of the estimators, as well as the respective oracles, are included in brackets Table \ref{tab:times_nodes}. Compared to the respective oracle, the early stopped estimators tend to stop slightly too late for the additive models, indicated by more terminal nodes.

%performance decrease due to estimated kappa:
When using the estimated noise level $\hat{\sigma}^2$ as the threshold $\kappa$, we observe a performance decrease of approximately 0–20\% for the early-stopped estimators. The upward bias in the noise level estimation error usually leads to slightly premature stopping.

% Small sample size:
In addition, we reduce the sample size of Simulation A from $n=1000$ to $n=100$ and show the results in Table \ref{tab:small_sample} in the \hyperref[simulation]{Appendix}. In this small sample setting, the early stopped methods outperform the cost-complexity pruning. Since pruning relies on cross-validation, it lacks observations to find the optimal tree structure. In comparison, the early stopping still works relatively well.
% d=10 and d=100:
Finally, we run Simulation A with $d=10$ and Simulation B with $d=100$ in Table \ref{tab:high_dim} in the \hyperref[simulation]{Appendix}. The results are similar to the initial simulation setting and, thus, underline the feasibility of early stopping for high dimensions.

%Reference to oracle comparison in appendix:
Further, the choice of the tree-growing mechanism influences the best achievable performance of the estimator. An analysis of the underlying oracle errors, detailed in Appendix \ref*{SecOrComp}, shows that the semi-global estimator often provides a better oracle, particularly for the rectangular and circular functions.

\begin{table}[tbp]
\spacingset{1.2}
\centering
\footnotesize
\begin{tabular}{lllllll}
\hline
 & Pruning & Global & Two-Step & Semi & Deep \\ \hline
Rectangular & 12.41s (6/6) & 0.53s (14/4) & 0.59s (5/7) & 0.54s (8.5/4) & 1.83s \\ 
Circular & 12.32s (5/5) & 0.53s (15/15) & 0.60s (5/6) & 0.45s (7.5/4) & 1.85s \\ 
Sine cosine & 13.32s (3/3) & 0.30s (4/4) & 0.35s (3/4) & 0.41s (3/3) & 1.60s \\ 
Elliptical & 16.30s (81/76) & 1.15s (305/536) & 2.85s (71/83) & 1.32s (105/143) & 1.66s \\ 
Additive smooth & 41.79s (26/27) & 6.32s (281.5/62.5) & 14.27s (26/33) & 7.39s (155/107) & 9.29s \\ 
Additive step & 39.09s (38/39) & 5.73s (202/64) & 12.59s (37/46) & 7.11s (118/88) & 8.46s \\ 
Additive linear & 38.96s (25/25) & 5.63s (211/32) & 12.68s (24/33) & 6.35s (134/48) & 8.83s \\ 
Additive hills & 34.30s (9/9) & 5.18s (96/16) & 7.74s (9/12) & 5.11s (115/4) & 9.12s \\ 
\hline
\end{tabular}
\vspace{-0.3cm}
\caption{Median run times in seconds (with median number of terminal nodes/median oracle number of terminal nodes in brackets). The deep tree has $n=1000$ terminal nodes.}
\label{tab:times_nodes}
\end{table}

\section{Conclusion}
\label{sec:discussion}
This paper introduces a novel, data-driven stopping rule for regression trees based on the discrepancy principle, offering a computationally efficient alternative to existing post-pruning and pre-pruning methods. Clear theoretical guarantees in terms of oracle inequalities are provided to understand how far early stopping is away from the oracle estimator along the tree growth.  The generalized projection flow framework further unifies and widens the scope of the proposed early stopping methodology. The oracle errors along breadth-first (global), best-first (semi-global) and pruning-type tree building are still the subject of intense current research, depending on specific splitting criteria and function classes. Also given our comparison of the oracle errors among the different methods, progress in the error analysis is highly desirable and would immediately give rise to complete risk bounds for early stopping.  First empirical results are quite convincing, and given the simplicity of the early stopping methods, wider practical experience is clearly within reach.

\section*{Acknowledgements}
This research has been partially funded by the Deutsche Forschungsgemeinschaft (DFG) – Project-ID 318763901 - SFB1294, Project-ID 460867398 - Research Unit 5381 and the German Academic Scholarship Foundation.

%\appendix
\clearpage
\centerline{\textbf{SUPPLEMENTAL: APPENDICES}}
  \medskip 
\bigskip
\bigskip

\spacingset{1.5} % DON'T change the spacing!

This supplementary document is organized in the following way. In Section \ref{proofs} we provide proofs of Lemma \ref{LemRiskBound} in Section \ref{proof:LemRiskBound} and Lemma \ref{LemRt2} in Section \ref{proof:LemRt2}. An additional result is given in Lemma \ref{app:LemOracle}. Further, we give proofs of Proposition \ref{PropDistanceOracle} in \ref{app:PropDistanceOracle}, Theorem \ref{ThmOI1} in \ref{app:ThmOI1}, Lemma \ref{LemIntPol} in \ref{app:LemIntPol}, Proposition \ref{PropCTIndep} in \ref{proof:ct_ind}, Proposition \ref{PropCTGeneral} in \ref{proof:ct_dep}, and Theorem \ref{ThmSGind} in \ref{app:ThmSGind}, Theorem \ref{ThmGind} in \ref{app:ThmGind} and Theorem \ref{ThmGdep} in \ref{app:ThmGdep}.

In Section \ref{simulation} we give further details on the simulation and include supplementary tables and figures of the simulation.

% In Section \ref{proofs} we provide proofs of Proposition 1, Theorem 2, Proposition 3 and Theorem 5. Intuition on the proposed bootstrap procedure for finite samples is given. In Section \ref{simulation} we give details on the simulation procedure and include supporting tables and figures of the simulation exercise as well as the empirical application.

%\clearpage
\appendix
\renewcommand\theequation{S\arabic{equation}}

\section{Proofs and supporting results}
\label{proofs}

\subsection{Proof of Lemma \ref{LemRiskBound} (error decomposition)}
\label{proof:LemRiskBound}
\begin{proof}
By insertion of $Y=f+\eps$ we obtain
\[ \norm{\hat F_t-f}_n^2=\norm{ \Pi_t(Y-f)-(\Id-\Pi_t)f}_n^2=\norm{\Pi_t\eps}_n^2+\norm{(\Id-\Pi_t)f}_n^2 -2\scapro{\Pi_t\eps}{(\Id-\Pi_t)f}_n.
\]
This is the claimed identity, and the upper bound follows by triangle inequality along the same lines.

The values of approximation and stochastic error at $t\in\{0,n\}$ are obtained from $\Pi_0=0$ and $\Pi_n=\Id$, respectively, while their continuity in $t$ is a consequence of the continuity of $t\mapsto \Pi_t$. The monotonicity of $\norm{\Pi_t\eps}_n^2$ follows from $\Pi_s\preceq\Pi_t$ and their commutativity via
\[ \norm{\Pi_s\eps}_n^2=\scapro{\Pi_s(\Pi_s^{1/2}\eps)}{\Pi_s^{1/2}\eps}_n\le \scapro{\Pi_t(\Pi_s^{1/2}\eps)}{\Pi_s^{1/2}\eps}_n
=\scapro{\Pi_s(\Pi_t^{1/2}\eps)}{\Pi_t^{1/2}\eps}_n\le \norm{\Pi_t\eps}_n^2.
\]
The monotonicity argument for $\norm{(\Id-\Pi_t)f}_n^2$ is completely analogous.
\end{proof}

\subsection{Proof of Lemma \ref{LemRt2} (non-increasing residuals)}
\label{proof:LemRt2}

\begin{proof}
We calculate
\begin{align*}
\norm{\hat F_t-\hat F_s}_n^2 &= \norm{(\Id-\Pi_s)Y-(\Id-\Pi_t)Y}_n^2\\
&= R_s^2+R_t^2-2\scapro{(\Id-\Pi_s)Y}{(\Id-\Pi_t)Y}_n\\
&= R_s^2-R_t^2-2\scapro{(\Pi_t-\Pi_s)Y}{(\Id-\Pi_t)Y}_n.
\end{align*}
Since $\Pi_t$ and $\Pi_s$ commute and $\Pi_s\preceq\Pi_t$ holds, the operators $\Pi_t-\Pi_s$, $\Id-\Pi_t$ are commuting, positive semi-definite and we obtain
\[ \scapro{(\Pi_t-\Pi_s)Y}{(\Id-\Pi_t)Y}_n=\scapro{(\Pi_t-\Pi_s)((\Id-\Pi_t)^{1/2}Y)}{(\Id-\Pi_t)^{1/2}Y}_n
\ge 0.
\]
Hence, $\norm{\hat F_t-\hat F_s}_n^2\le R_s^2-R_t^2$ follows, in particular $R_s^2-R_t^2\ge 0$. The continuity of $t\mapsto R_t^2$ follows from the continuity of $t\mapsto\Pi_tY$.
\end{proof}

\subsection*{Lemma Oracle-type property}
\label{app:LemOracle}

\begin{lemma}\label{LemOracle}
At the balanced oracle, the loss satisfies the oracle-type property
\[ \norm{\hat F_{\tau_b}-f}_n^2\le 2\norm{(\Id-\Pi_{\tau_b})f}_n^2+2\norm{\Pi_{\tau_b}\eps}_n^2\le 4\inf_{t\in[0,n]}\Big(\norm{(\Id-\Pi_t)f}_n^2+\norm{\Pi_t\eps}_n^2\Big).\]
\end{lemma}

\begin{proof}
By the monotonicity properties of the error terms we have $\norm{(\Id-\Pi_t)f}_n\ge \norm{(\Id-\Pi_{\tau_b})f}_n$ for $t\le \tau_b$ and $\norm{\Pi_t\eps}_n\ge \norm{\Pi_{\tau_b}\eps}_n$ for $t\ge\tau_b$. Consequently, for any $t\in[0,n]$
\begin{align*}
\norm{(\Id-\Pi_t)f}_n^2+\norm{\Pi_t\eps}_n^2 &\ge \max\Big(\norm{(\Id-\Pi_t)f}_n^2,\norm{\Pi_t\eps}_n^2\Big)\\
&\ge \min\Big(\norm{(\Id-\Pi_{\tau_b})f}_n^2,\norm{\Pi_{\tau_b}\eps}_n^2\Big)\\
&=\frac12 \Big(\norm{(\Id-\Pi_{\tau_b})f}_n^2+\norm{\Pi_{\tau_b}\eps}_n^2\Big)
\end{align*}
holds. By Lemma \ref{LemRiskBound}, the last line is larger than $\frac14 \norm{\hat F_{\tau_b}-f}_n^2$, and the result follows by taking the infimum over all $t\in[0,n]$.
\end{proof}

\subsection{Proof of Proposition \ref{PropDistanceOracle} (distance to balanced oracle)}
\label{app:PropDistanceOracle}
\begin{proof}
For $\kappa\in[0,\norm{Y}_n^2]$ we have $R_\tau^2=\kappa$ and we obtain by Lemma \ref{LemRt2}
\begin{align*}
\norm{\hat F_\tau-\hat F_{\tau_b}}_n^2 &\le \abs{R_\tau^2-R_{\tau_b}^2}\\
&=\babs{\kappa-\norm{(\Id-\Pi_{\tau_b})(f+\eps)}_n^2}\\
&=\babs{\kappa-\norm{(\Id-\Pi_{\tau_b})f}_n^2-\norm{(\Id-\Pi_{\tau_b})\eps}_n^2
-2\scapro{(\Id-\Pi_{\tau_b})f}{(\Id-\Pi_{\tau_b})\eps}_n}\\
&=\babs{\kappa-\norm{\eps}_n^2+2\scapro{(\Pi_{\tau_b}-\Pi_{\tau_b}^2)\eps}{\eps}_n-2\scapro{(\Id-\Pi_{\tau_b})^2f}{\eps}_n},
\end{align*}
where $\norm{(\Id-\Pi_{\tau_b})f}_n^2=\norm{\Pi_{\tau_b}\eps}_n^2$ for the balanced oracle entered in the last line.  For $\kappa>\norm{Y}_n^2$ we stop at $\tau=0$ such that
\[0\le R_\tau^2-R_{\tau_b}^2=\norm{Y}_n^2-R_{\tau_b}^2<\kappa-R_{\tau_b}^2\]
holds. Hence, in this case, the first identity in the equation array above becomes an inequality, and the same upper bound holds. It remains to apply the triangle inequality.
\end{proof}

\subsection{Proof of Theorem \ref{ThmOI1} (general oracle inequality)}
\label{app:ThmOI1}
\begin{proof}
Since all $\Pi_t$ commute, there is one orthogonal matrix $O$, only depending on the projection flow, such that
\begin{equation}\label{EqDiag}
D_t:=O \Pi_t O^\top\text{ is diagonal  for all $t$.}
\end{equation}
For the interpolation term, we introduce $\tilde\eps=O\eps$ and obtain by diagonalization
\begin{align*} 
\scapro{(\Pi_{\tau_b}-\Pi_{\tau_b}^2)\eps}{\eps}_n&=\scapro{(D_{\tau_b}-D_{\tau_b}^2)\tilde\eps}{\tilde\eps}_n
=\frac1n\sum_{i=1}^n (D_{\tau_b}-D_{\tau_b}^2)_{ii}\tilde\eps_i^2\\
&\le \trace(D_{\tau_b}-D_{\tau_b}^2)n^{-1}\max_{i=1,\ldots,n}\tilde\eps_i^2.
\end{align*}
The Legendre transformation  of $g(v)=e^v$ is $g^\ast(u)=u\log(u)-u$, that is $uv\le e^v+u\log(u)-u$ holds for $u,v\ge 0$. For nonnegative random variables $U,V$ with $\E[U\log(U)]<\infty$ and $\E[e^{\alpha V}]<\infty$ for some $\alpha>0$ we insert $u=\frac{U}{\E[U]}$, $v=\alpha V-\log(\E[e^{\alpha V}])$ and take expectations to infer
\[ \alpha\E[UV]\le \log(\E[e^{\alpha V}])\E[U]+\E[U\log(U)]-\log(\E[U])\E[U].\]
By the $\bar\sigma$-subgaussianity of $\eps$ also $\tilde\eps=O\eps$ is $\bar\sigma$-subgaussian and we deduce that \[\E\Big[\exp\Big(\alpha \max_{1\le i\le n}\tilde\eps_i^2/\bar\sigma^2\Big)\Big]\le \sum_{i=1}^n\E\big[\exp\big(\alpha \tilde\eps_i^2/\bar\sigma^2\big)\big]\le 2n\text{ for }\alpha=1/8,
\]
by \citet{boucheron2013concentration}, Eq. (2.4), and thus
\[ \tfrac18\E\Big[U\max_{i=1,\ldots,n}\tilde\eps_i^2\bar\sigma^{-2}\Big]\le \log(2n)\E[U]+\E[U\log(U)]-\log(\E[U])\E[U].\]
For $U=\trace(D_{\tau_b}-D_{\tau_b}^2)=\trace(\Pi_{\tau_b}-\Pi_{\tau_b}^2)$, we always have $U\log(U)\le\log(n)U$ whence
\begin{align*} 
\E[\scapro{(\Pi_{\tau_b}-\Pi_{\tau_b}^2)\eps}{\eps}_n] &\le \E\Big[n^{-1}\trace(\Pi_{\tau_b}-\Pi_{\tau_b}^2)
\max_{i=1,\ldots,n}\tilde\eps_i^2\Big]\\
&\le 16\bar\sigma^2\log(2n) n^{-1}\E[\trace(\Pi_{\tau_b}-\Pi_{\tau_b}^2)]
\end{align*}
and we arrive at \eqref{EqOI1a}.

To establish \eqref{EqOI1b},  we bound for any $\delta>0$, using $2AB\le\delta A^2+\delta^{-1}B^2$,
\[ 2\E\big[\abs{\scapro{(\Id-\Pi_{\tau_b})^2f}{\eps}_n}\big]\le \delta \E\big[\norm{(\Id-\Pi_{\tau_b})f}_n^2\big]
+\delta^{-1}\E\big[\scapro{\tfrac{(\Id-\Pi_{\tau_b})^2f}{\norm{(\Id-\Pi_{\tau_b})f}_n}}{\eps}_n^2\big],
\]
which holds in particular also under convention \eqref{EqOI1Conv}.
Therefore, we find by triangle inequality and Lemma \ref*{LemOracle}
\begin{align*}
\E[\norm{\hat F_\tau-f}_n^2] &\le 2\E[\norm{\hat F_{\tau_b}-f}_n^2]+2\E[\norm{\hat F_\tau-\hat F_{\tau_b}}_n^2]\\
&\le (4+2\delta)\E[\norm{(\Id-\Pi_{\tau_b})f}_n^2]+4\E[\norm{\Pi_{\tau_b}\eps}_n^2]+2\E[\abs{\kappa-\norm{\eps}_n^2}]
\\
&\quad + 32\bar\sigma^2\frac{\log(2n)}{n}\E[\trace(\Pi_{\tau_b}-\Pi_{\tau_b}^2)] +2\delta^{-1}\E\big[\scapro{\tfrac{(\Id-\Pi_{\tau_b})^2f}{\norm{(\Id-\Pi_{\tau_b})f}_n}}{\eps}_n^2\big].\\
&\le (8+2\delta)\E\Big[\inf_{t\in[0,n]}\Big(\norm{(\Id-\Pi_{t})f}_n^2+\norm{\Pi_{t}\eps}_n^2\Big)\Big]
+2\E[\abs{\kappa-\norm{\eps}_n^2}]\\
&\quad + 32\bar\sigma^2\frac{\log(2n)}{n}\E[\trace(\Pi_{\tau_b}-\Pi_{\tau_b}^2)] +2\delta^{-1}\E\big[\scapro{\tfrac{(\Id-\Pi_{\tau_b})^2f}{\norm{(\Id-\Pi_{\tau_b})f}_n}}{\eps}_n^2\big].
\end{align*}
Inequality \eqref{EqOI1b} follows with the choice $\delta=1/2$.
\end{proof}

\subsection{Proof of Lemma \ref{LemIntPol} (interpolation error bounds)}
\label{app:LemIntPol}

\begin{proof}
In the semi-global setting of Example \ref{Ex1}, it suffices to note that
\begin{equation}\label{EqIntAlgebra} \Pi_t-\Pi_t^2=((1-\alpha)\Pi_k+\alpha \Pi_{k+1})-((1-\alpha^2)\Pi_k+\alpha^2\Pi_{k+1})=\alpha(1-\alpha)(\Pi_{k+1}-\Pi_k),
\end{equation}
implying $\sup_t\trace(\Pi_t-\Pi_t^2)\le \frac14$.
In the global case, we have
$\Pi_t-\Pi_t^2=\alpha(1-\alpha)(\Pi_{k_{g(t)+1}}- \Pi_{k_{g(t)}})$ and $k_{g(t)+1}\le 2k_{g(t)}\vee 1$.
This gives
\begin{equation*}
    \trace(\Pi_t-\Pi_t^2)=\alpha(1-\alpha)(k_{g(t)+1}-k_{g(t)})\le (k_{g(t)}\vee 1)/4\le (\trace(\Pi_t^2)\vee 1)/4.
\end{equation*}
With $\norm{\Pi_t}_{HS}^2=\trace(\Pi_t^2)$ the result follows.
\end{proof}

\subsection{Proof of Proposition \ref{PropCTIndep} (cross-term, independent case)}
\label{proof:ct_ind}

\begin{proof}
By definition, $\tau_b$ only depends on $(\norm{\Pi_t\eps}_n^2)_{t\in[0,n]}$. In view of the diagonalisation $\Pi_t=O^\top D_t O$ from \eqref{EqDiag}, $\norm{\Pi_t\eps}_n^2=\norm{D_tO\eps}_n^2$ only depends on $(\tilde\eps_i^2)_{i=1,\ldots,n}$ with $\tilde\eps:=O\eps\sim N(0,\sigma^2\Id)$ and on the design $(X_i)$, treating $\Pi_t$ and thus $D_t,O$ as fixed by independence. This shows for the conditional expectation
\begin{align*}
\E\big[\scapro{(\Id-\Pi_{\tau_b})^2f}{\eps}_n^2\,&\big|\,(X_i,\tilde\eps_i^2)_{i=1,\ldots,n}\big] = \E\big[\scapro{(\Id-D_{\tau_b})^2Of}{\tilde\eps}_n^2\,\big|\,(X_i,\tilde\eps_i^2)_{i=1,\ldots,n}\big]\\
&= \E\Big[\Big(\frac1{n}\sum_{i=1}^n ((\Id-D_{\tau_b})^2Of)(X_i)\tilde\eps_i\Big)^2\,\Big|\,(X_i,\tilde\eps_i^2)_{i=1,\ldots,n}\Big]\\
&= \frac1{n^2}\sum_{i=1}^n ((\Id-D_{\tau_b})^2Of)(X_i)^2\tilde\eps_i^2.\\
&\le \norm{(\Id-D_{\tau_b})^2Of}_n^2n^{-1}\max_{i=1,\ldots,n}\tilde\eps_i^2,
\end{align*}
where we used $\E[\tilde\eps_i\tilde\eps_j\,|\,(X_i,\tilde\eps_i^2)_{i=1,\ldots,n}]=0$ for $i\not=j$ due to $\tilde\eps\sim N(0,\sigma^2\Id)$ conditional on the design. Noting $\norm{(\Id-D_{\tau_b})^2Of}_n^2=\norm{(\Id-\Pi_{\tau_b})^2f}_n^2$, this shows for the full expectation
\[ \E\Big[\scapro{\tfrac{(\Id-\Pi_{\tau_b})^2f}{\norm{(\Id-\Pi_{\tau_b})f}_n}}{\eps}_n^2\Big]\le n^{-1}\E\Big[\tfrac{\norm{(\Id-\Pi_{\tau_b})^2f}_n^2}{\norm{(\Id-\Pi_{\tau_b})f}_n^2} \max_{i=1,\ldots,n}\tilde\eps_i^2\Big]\le n^{-1}\E\Big[ \max_{i=1,\ldots,n}\tilde\eps_i^2\Big].
\]
By a standard Gaussian maximal inequality \citep{tsybakov2009nonparametric}, Corollary 1.3, we have $\E[ \max_{i=1,\ldots,n}\tilde\eps_i^2]\le 4\sigma^2\log(\sqrt2 n)$ and the result follows.
\end{proof}

\subsection{Proof of Proposition \ref{PropCTGeneral} (cross-term, dependent case)}
\label{proof:ct_dep}

\begin{proof}
The proof relies on a deviation inequality for maxima of normalized quadratic functionals and a complexity bound for the number of possible projections until time $t$. 
% This is inspired by \citet{hucker2024early}, but the argument differs.

By the linear interpolation property $\Pi_{\tau_b}=\alpha\Pi_{k_{g(\tau_b)}}+(1-\alpha)\Pi_{k_{g(\tau_b)+1}}$ holds, where $\tau_b=(1-\alpha)k_{g(\tau_b)}+\alpha k_{g(\tau_b)+1}$ and $\alpha\in[0,1)$, see Example \ref{Ex1}. Simple algebra shows
\[ (\Id-\Pi_{\tau_b})^2=(1-\alpha)^2(\Id-\Pi_{k_{g(\tau_b)}})+(1-(1-\alpha)^2)(\Id-\Pi_{k_{g(\tau_b)+1}}).\]
This yields
\begin{align*}
\abs{\scapro{(\Id-\Pi_{\tau_b})^2f}{\eps}_n} &\le (1-\alpha)^2\norm{(\Id-\Pi_{k_{g(\tau_b)}})f}_n\babs{\scapro{\tfrac{(\Id-\Pi_{k_{g(\tau_b)}})f}
{\norm{(\Id-\Pi_{k_{g(\tau_b)}})f}_n}}{\eps}_n}\\
&\quad +
(1-(1-\alpha)^2)\norm{(\Id-\Pi_{k_{g(\tau_b)+1}})f}_n
\babs{\scapro{\tfrac{(\Id-\Pi_{k_{g(\tau_b)+1}})f}{\norm{(\Id-\Pi_{k_{g(\tau_b)+1}})f}_n}}{\eps}_n}
\end{align*}
as well as
\begin{align*}
\norm{(\Id-\Pi_t)f}_n&=\scapro{(\Id-\Pi_t)^2f}{f}_n^{1/2}\\
&=\Big((1-\alpha)^2\norm{(\Id-\Pi_{k_{g(\tau_b)}})f}_n^2
+(1-(1-\alpha)^2)\norm{(\Id-\Pi_{k_{g(\tau_b)+1}})f}_n^2\Big)^{1/2}\\
&\ge (1-\alpha)^2\norm{(\Id-\Pi_{k_{g(\tau_b)}})f}_n
+(1-(1-\alpha)^2)\norm{(\Id-\Pi_{k_{g(\tau_b)+1}})f}_n,
\end{align*}
where Jensen's inequality was used in the last step. We deduce
\[ \scapro{\tfrac{(\Id-\Pi_{\tau_b})^2f}{\norm{(\Id-\Pi_{\tau_b})f}_n}}{\eps}_n^2\le \max\Big(\scapro{\tfrac{(\Id-\Pi_{k_{g(\tau_b)}})f}{\norm{(\Id-\Pi_{k_{g(\tau_b)}})f}_n}}{\eps}_n^2,
\scapro{\tfrac{(\Id-\Pi_{k_{g(\tau_b)+1}})f}{\norm{(\Id-\Pi_{k_{g(\tau_b)+1}})f}_n}}{\eps}_n^2\Big).
\]

Since $\Pi_{\tau_b}$ and $\eps$ have an almost arbitrary dependence structure, we bound the expectation by the maximum over all possible splits leading to $\Pi_{\tau_b}$:
\[ \E\Big[\scapro{\tfrac{(\Id-\Pi_{\tau_b})^2f}{\norm{(\Id-\Pi_{\tau_b})f}_n}}{\eps}_n^2\Big]\le \E\Big[\max_{\trace(\Pi)\in\{k_{g(\tau_b)},k_{g(\tau_b)+1}\}}\scapro{\tfrac{(\Id-\Pi)f}{\norm{(\Id-\Pi)f}_n}}{\eps}_n^2\Big],
\]
where the maximum is taken over all deterministic projections $\Pi$ reachable by data-driven $k_{g(\tau_b)}$-fold or $k_{g(\tau_b)+1}$-fold global and semi-global splitting, respectively.

We must bound the expectation of this maximum. Since $v:=\tfrac{(\Id-\Pi)f}{\norm{(\Id-\Pi)f}_n}$ is a unit vector, we get without the maximum  $\E[ \scapro{v}{\eps}_n^2]=\sigma^2 n^{-1}$. For $M_g:=\#\{\Pi: \trace(\Pi)=k_g\}$, the cardinality of the maximisation set at a deterministic generation $g$, we consider
$\max_{m=1,\ldots,M_g}\scapro{v_m}{\eps}_n^2$ generally for any deterministic unit vectors $v_m$.
By $\bar\sigma$-subgaussianity  we have
\[ \PP\Big(\max_{m=1,\ldots,M_g}\scapro{v_m}{\eps}_n^2> \tfrac{2\bar\sigma^2}{n}(\log(M_g)+u)\Big)\le e^{-u},\quad u\ge 0.\]
Let us determine an upper bound for $\log(M_g)$. Projections $\Pi$ with $\trace(\Pi)=k_g$ are obtained after $k_g$ splits.
At each split, we have at most $d(n-1)$ different thresholds for splitting in each of the $d$ directions and at all interstices between the $n$ data points. Therefore $M_g\le (dn)^{k_g}$, $\log(M_g)\le k_g\log(dn)$ holds, which does not seem a too rough upper bound at the logarithmic level.  We conclude
\[ \PP\Big(\max_{\trace(\Pi)=k_g} \frac{n\scapro{(\Id-\Pi)^2f}{\eps}_n^2}{2\bar\sigma^2\norm{(\Id-\Pi)^2f}_n^2}-k_g\log(dn)> u\Big)
\le e^{-u},\quad u\ge 0.
\]
Bonferroni's inequality and $\sum_{g=0}^{G-1}e^{-(u+\log(G))}=e^{-u}$ yield
\[ \PP\Big(\max_{g=0,\ldots,G-1}\Big(\max_{\trace(\Pi)=k_g} \frac{n\scapro{(\Id-\Pi)^2f}{\eps}_n^2}{2\bar\sigma^2\norm{(\Id-\Pi)^2f}_n^2}-k_g\log(dn)\Big)-\log(G)> u\Big)
\le e^{-u}.
\]
Integrating this bound over $u\in[0,\infty)$ gives
\[ \E\Big[\max_{g=0,\ldots,G}\max_{\trace(\Pi)=k_g} \Big( \frac{n\scapro{(\Id-\Pi)^2f}{\eps}_n^2}{2\bar\sigma^2\norm{(\Id-\Pi)^2f}_n^2}-k_g\log(dn)\Big)\Big]
\le \log(G)+1=\log(eG),
\]
noting $(\Id-\Pi)f=0$ for $\trace(\Pi)=k_G=n$, i.e. $\Pi=\Id$, such that the maximum trivially extends to $g=G$.
The uniformity over $g$ then shows
\[ \E\Big[\max_{\trace(\Pi)\in\{k_{g(\tau_b)},k_{g(\tau_b)+1}\}} \frac{\scapro{(\Id-\Pi)^2f}{\eps}_n^2}{\norm{(\Id-\Pi)^2f}_n^2}\Big]\le \frac{2\bar\sigma^2}{n}\Big(\E[k_{g(\tau_b)+1}]\log(dn)+\log(eG)\Big).
\]
For both semi-global and global early stopping, we have $G\le n$ and $k_{g(\tau_b)+1}\le 2\tau_b\vee 1$ such that
the result follows by simplifying the numerical constants.
\end{proof}

\par\medskip % Adds a bit of vertical space before
\noindent\textbf{Remark.} % Creates the bold title
\textit{The maximal deviation inequality in our argument does not profit from any correlation between $\scapro{(\Id-\Pi)^2f}{\eps}_n$ for different projections $\Pi$. The problem is that for $\Pi,\Pi'$, where $\norm{(\Id-\Pi)^2f-(\Id-\Pi')^2f}_n$ is small, the intrinsic normalised distance $\norm{\frac{(\Id-\Pi)^2f}{\norm{(\Id-\Pi)^2f}_n}-\frac{(\Id-\Pi')^2f}{\norm{(\Id-\Pi')^2f}_n}}_n$ might still be large because the approximation-type error $\norm{(\Id-\Pi)^2f}_n$ in the denominator is typically small.}\\
\textit{
Only in the case $d=1$, a partial summation (or integration by parts) argument for bounded variation functions $f$ yields a useful cross-term bound of order $\bar\sigma^2 n^{-1/2}$. This improves the bound of Proposition \ref{PropCTGeneral} for all $\tau_b\ge n^{1/2}$ and is not larger than the typical order of the early stopping error. Unfortunately, the novel approach by \citet{klusowski2023large} for analyzing the impurity gain in additive models of bounded variation seems to give only suboptimal bounds for our cross-term. So, improvements in the case $d\ge 2$ of interest remain open.}
\par\medskip % Adds a bit of vertical space after

\subsection{Proof of Theorem \ref{ThmSGind} (semi-global early stopping, independent splitting)}
\label{app:ThmSGind}

\begin{proof}
Insert into the oracle inequality of Theorem \ref{ThmOI1} the semi-global interpolation error bound of Lemma \ref{LemIntPol} and the cross-term bound of Proposition \ref{PropCTIndep} and note
\begin{align*} 
\E\Big[\inf_{t\in[0,n]}\Big(\norm{(\Id-\Pi_{t})f}_n^2+\norm{\Pi_{t}\eps}_n^2\Big)\Big]&\le \inf_{t\in[0,n]}\E\Big[\norm{(\Id-\Pi_{t})f}_n^2+\norm{\Pi_{t}\eps}_n^2\Big]\\
&=\inf_{t\in[0,n]}\E[\norm{\hat F_t-f}_n^2],
\end{align*}
by the independence of $\Pi_t$ from noise and design.
\end{proof}

\subsection{Proof of Theorem \ref{ThmGind} (global early stopping, independent splitting)}
\label{app:ThmGind}

\begin{proof}
The proof is as for the semi-global case in Theorem \ref{ThmSGind}, but using the global interpolation error bound of Lemma \ref{LemIntPol}.
\end{proof}

\subsection{Proof of Theorem \ref{ThmGdep} (early stopping, dependent splitting)}
\label{app:ThmGdep}

\begin{proof}
Insert into the oracle inequality of Theorem \ref{ThmOI1} the global interpolation error bound of Lemma \ref{LemIntPol}, also valid in the semi-global case, and the cross-term bound of Proposition \ref{PropCTGeneral}. Note
 $\norm{\Pi_{\tau_b}}_{HS}^2\le \tau_b$ and simplify the constants.
\end{proof}

\section{Further simulation details}
\label{simulation}

\subsection{Post-pruning tuning}\label{app:pruning}
% Description of Setup
For cost-complexity pruning, we first generate the sequence of cost-complexity hyperparameters $(\lambda_h)_{h=1}^{H}$, for $\lambda_h \in \Lambda_{prun}$. We then discard any hyperparameter $\lambda_h$ where $R(T_{\lambda_h}) - R(T_{\lambda_{h-1}}) > 0.01$ for $h=1,\ldots,H$, along with its corresponding subtrees $T_{\lambda_h}$. This additional step avoids a significant increase in computational cost from applying cross-validation across the full sequence $(\lambda_h)_{h=1}^{H}$, without sacrificing performance\footnote{This step is heuristically motivated, since the full path of cost-complexity hyperparameters can potentially include many elements and thus hinder practical feasibility of pruning. In contrast, the early stopping does not face this issue.}. The same approach is applied within the two-step procedure, where pruning is performed on the tree of depth $\hat g + 1$.

\subsection{Robustness}
Figure \ref{fig:additive_dgp} displays the additive functions used in Simulation B. The results for $n=100$ are included in Table \ref{tab:small_sample}. In Table \ref{tab:high_dim} the dimension is increased to $d=10$ for Simulation A and to $d=100$ for Simulation B.

\begin{figure}[htb]
    \centering
    \begin{subfigure}[b]{0.4\textwidth}
        \centering
        \includegraphics[width=\textwidth]{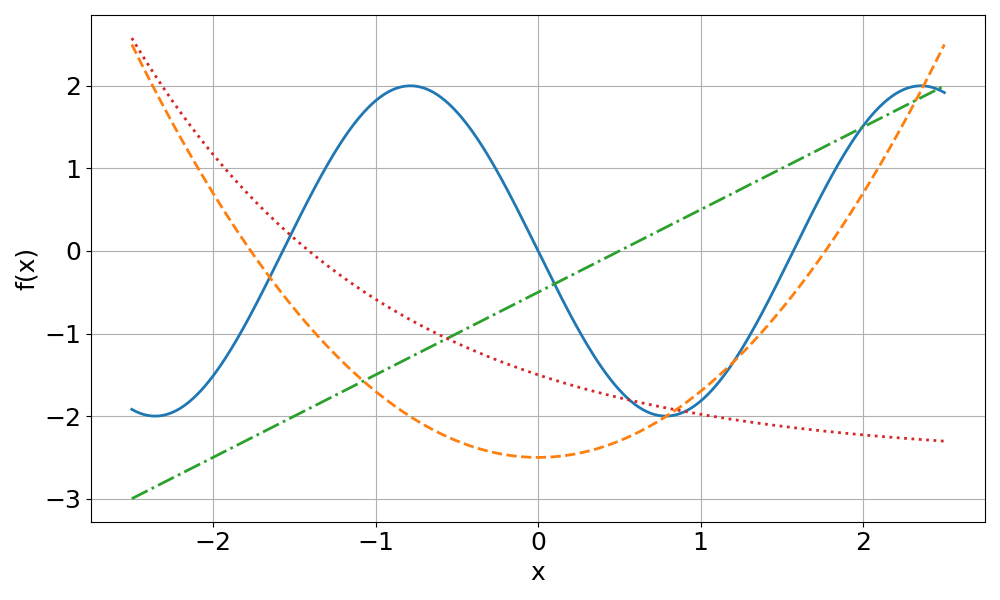}
        \caption{Smooth functions}
    \end{subfigure}
    \hspace{-0.2cm}
    \begin{subfigure}[b]{0.4\textwidth}
        \centering
        \includegraphics[width=\textwidth]{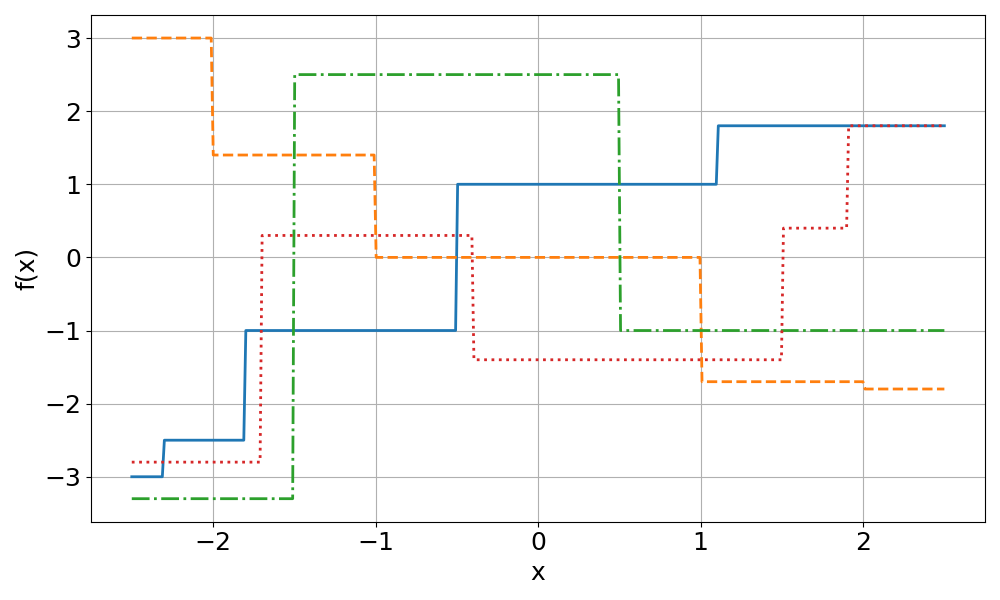}
        \caption{Step functions}
    \end{subfigure}\\
    \vspace{0mm} % Spacing between the rows
    \begin{subfigure}[b]{0.4\textwidth}
        \centering
        \includegraphics[width=\textwidth]{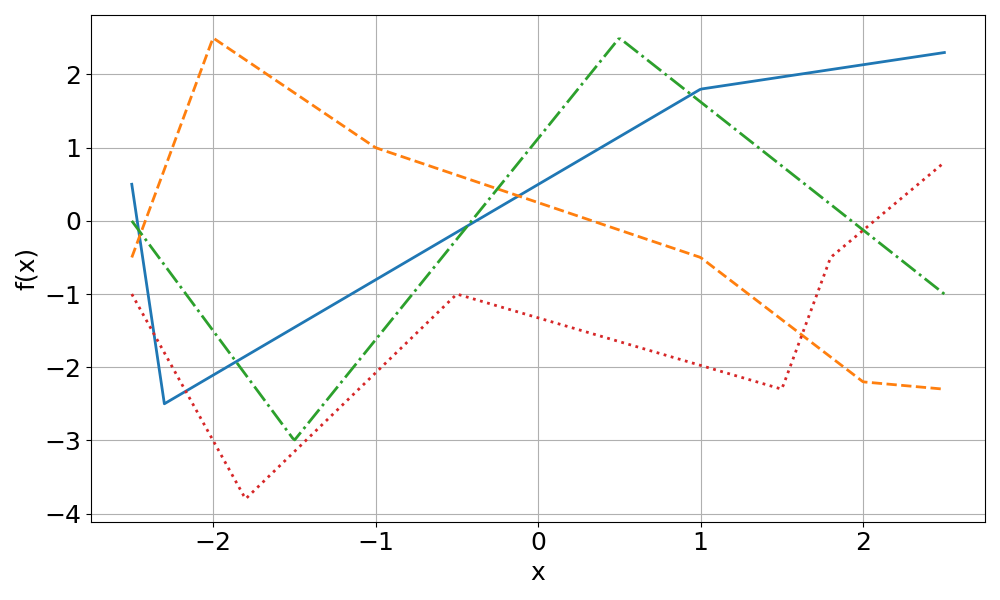}
        \caption{Piecewise linear functions}
    \end{subfigure}
    \hspace{-0.2cm}
    \begin{subfigure}[b]{0.4\textwidth}
        \centering
        \includegraphics[width=\textwidth]{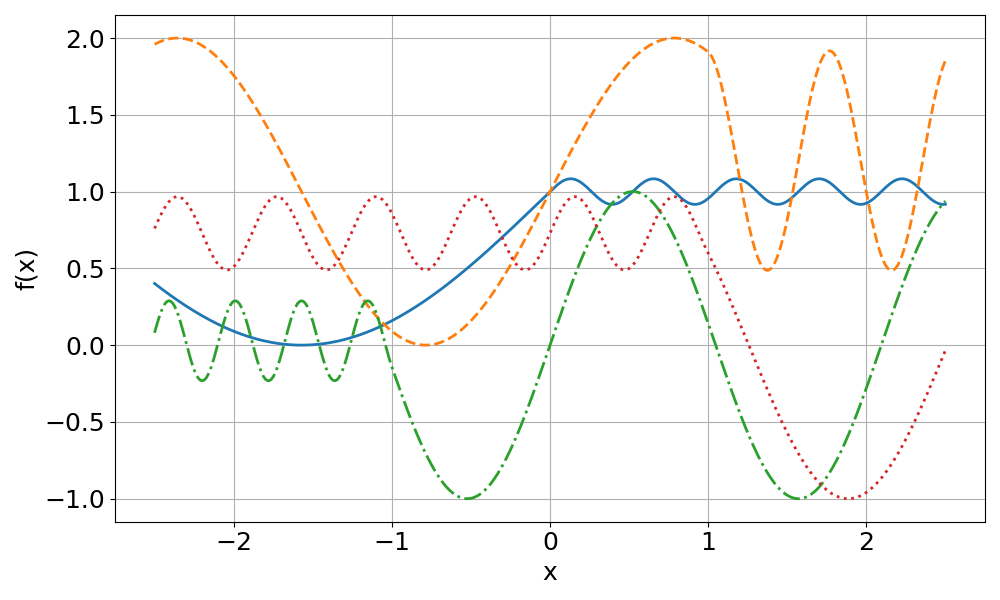}
        \caption{Hills-type functions}
    \end{subfigure}
    \caption{Signals $g_j$, $j=1,\ldots,4$, in the sparse high-dimensional additive model.}
    \label{fig:additive_dgp}
\end{figure}

\begin{table}[htb]
\centering
\footnotesize
\begin{tabular}{lcccccccc}
\hline
& Pruning & Global & Global Int & Two-Step & Semi\\
\hline
Rectangular & 0.57 & 0.48  & 0.38 & 0.47 & 0.47 \\
Circular & 0.58  & 0.57 & 0.47 & 0.53  & 0.55 \\
Sine cosine & 0.44  & 0.34 & 0.29  & 0.37 & 0.37 \\
Elliptical & 2.60  & 2.59 & 2.59 & 2.60 & 2.58  \\
\hline
\end{tabular}
\caption{Median RMSE for $n=100$ and $M=500$ Monte Carlo iterations.}
\label{tab:small_sample}
\end{table}

\begin{table}[htb]
\centering
\footnotesize
\begin{tabular}{lccccc}
\hline
 & Pruning & Global & Global Int & Two-Step & Semi \\ \hline
Rectangular & 0.21 & 0.35 & 0.32 & 0.25 & 0.35 \\
Circular & 0.25 & 0.38 & 0.36 & 0.25 & 0.33 \\
%Sine Cosine (old example) & 0.63 & 0.77 & 0.74 & 0.65 & 0.89 \\
Sine cosine & 0.23 & 0.22 & 0.21 & 0.21 & 0.23 \\
Elliptical & 1.18 & 1.36 & 1.37 & 1.23 & 1.27 \\
Additive smooth & 1.80 & 2.12 & 2.06 & 1.80 & 2.14 \\
Additive step & 1.52 & 1.74 & 1.66 & 1.53 & 1.72 \\
Additive linear & 1.59 & 1.88 & 1.79 & 1.59 & 1.86 \\
Additive hills & 0.86 & 1.12 & 1.04 & 0.87 & 1.18 \\
\hline
\end{tabular}
\caption{Median RMSE on the test set with $\kappa=\sigma^2$ with $d=10$ for Simulation A and $d=100$ for Simulation B.}
\label{tab:high_dim}
\end{table}

\subsection{Oracle comparison}\label{SecOrComp}

%Intro

\begin{comment}
Overall, the early-stopped oracles achieve comparable performance as the pruned oracles, especially for high-dimensional functions. Given that the pruned oracle has access to all splits in a fully grown tree, it is expected to slightly outperform an early stopped estimator. It comes at the cost of computational efficiency, as pruning operates in a bottom-up fashion.
%Relate it to kernel regression
This finding aligns with analogous discussions in other areas of statistical estimation. For instance, in kernel regression, the choice between bottom-up and top-down bandwidth selection methods has similar implications. \citet{lepski1997optimal} propose starting with a small bandwidth and increasing it, akin to a bottom-up approach. Similarly, \citet{blanchard2018early} establish that top-down methods, like early stopping, inherently incur an additional error compared to bottom-up methods.
\end{comment}

% In contrast, the early stopping methods, by design, cannot consider splits that occur at deeper levels since they halt the growth prematurely. There may be scenarios where an important split occurs at a deeper level, but reaching it requires making several less significant splits. The pruning method is better equipped to handle such situations, as it can prune away unnecessary branches while retaining crucial splits, leading to improved prediction performance.

\begin{figure}[tbp]
    \centering
   \begin{subfigure}[b]{0.4\textwidth}
        \centering
        \includegraphics[width=\textwidth]{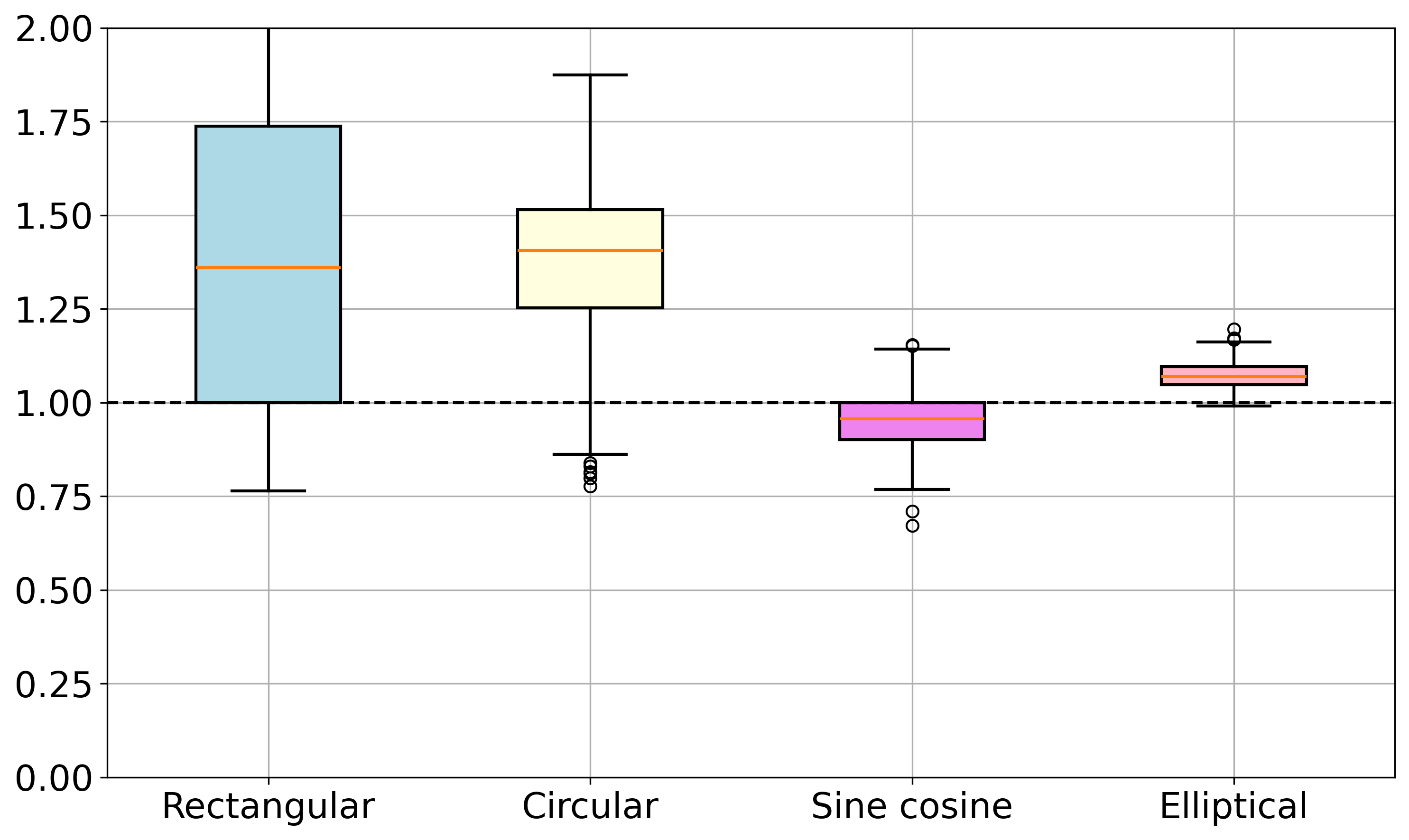}
        \caption{$\rho_{glob,semi}$ for Simulation A}
    \end{subfigure}
    \hspace{-0.2cm}
    \begin{subfigure}[b]{0.4\textwidth}
        \centering
        \includegraphics[width=\textwidth]{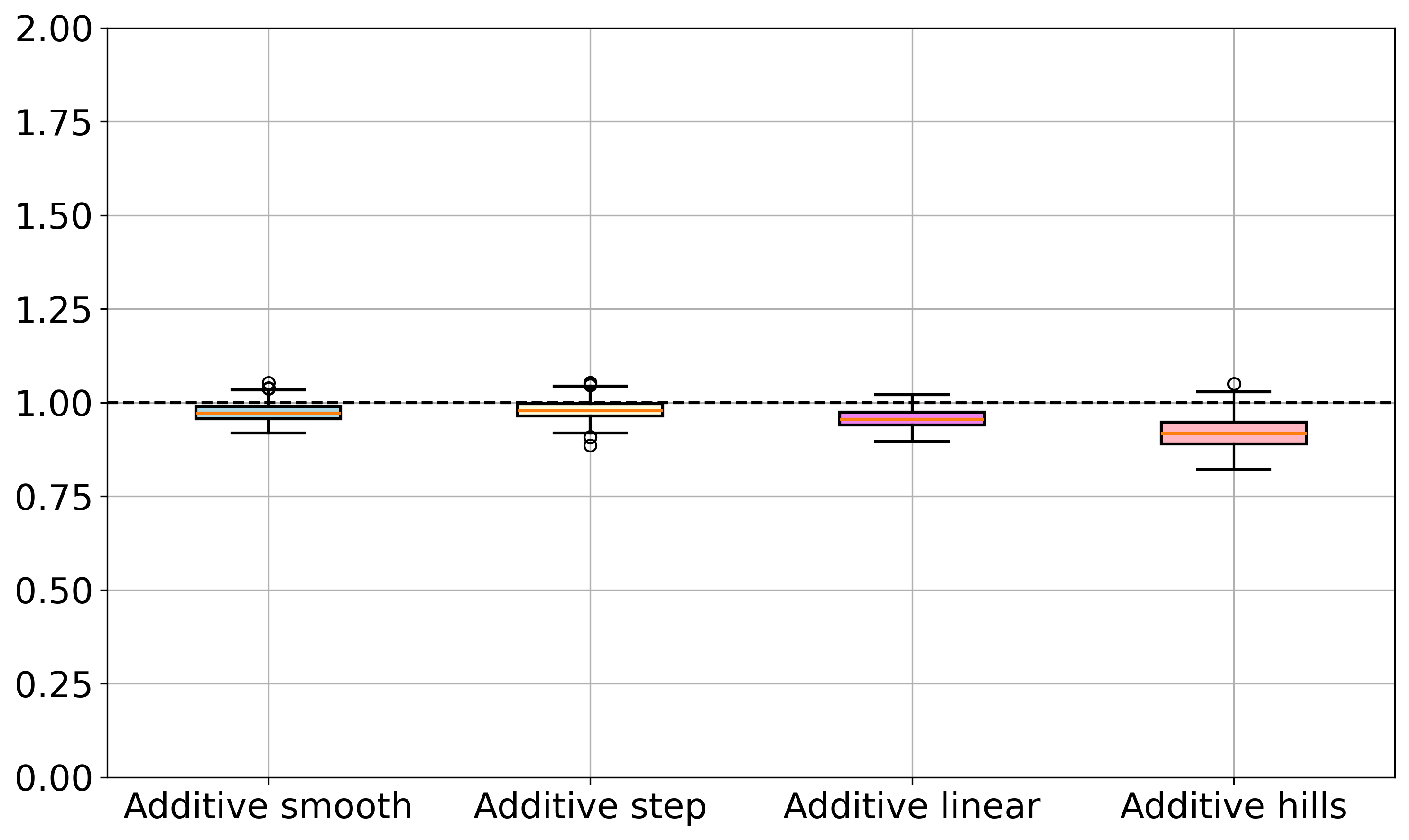}
        \caption{$\rho_{glob,semi}$ for Simulation B}
    \end{subfigure}
  %   \begin{subfigure}[b]{0.4\textwidth}
  %      \centering
  %      \includegraphics[width=\textwidth]{ratio_boxplot_ratio1.png}
  %      \caption{$\rho_{pg}$ for Simulation A}
  %  \end{subfigure}
  %  \hspace{-0.2cm}
  %  \begin{subfigure}[b]{0.4\textwidth}
  %      \centering
  %      \includegraphics[width=\textwidth]{ratio_boxplot_additive_ratio1.png}
  %      \caption{$\rho_{pg}$ for Simulation B}
  %  \end{subfigure}
    \begin{subfigure}[b]{0.4\textwidth}
        \centering
        \includegraphics[width=\textwidth]{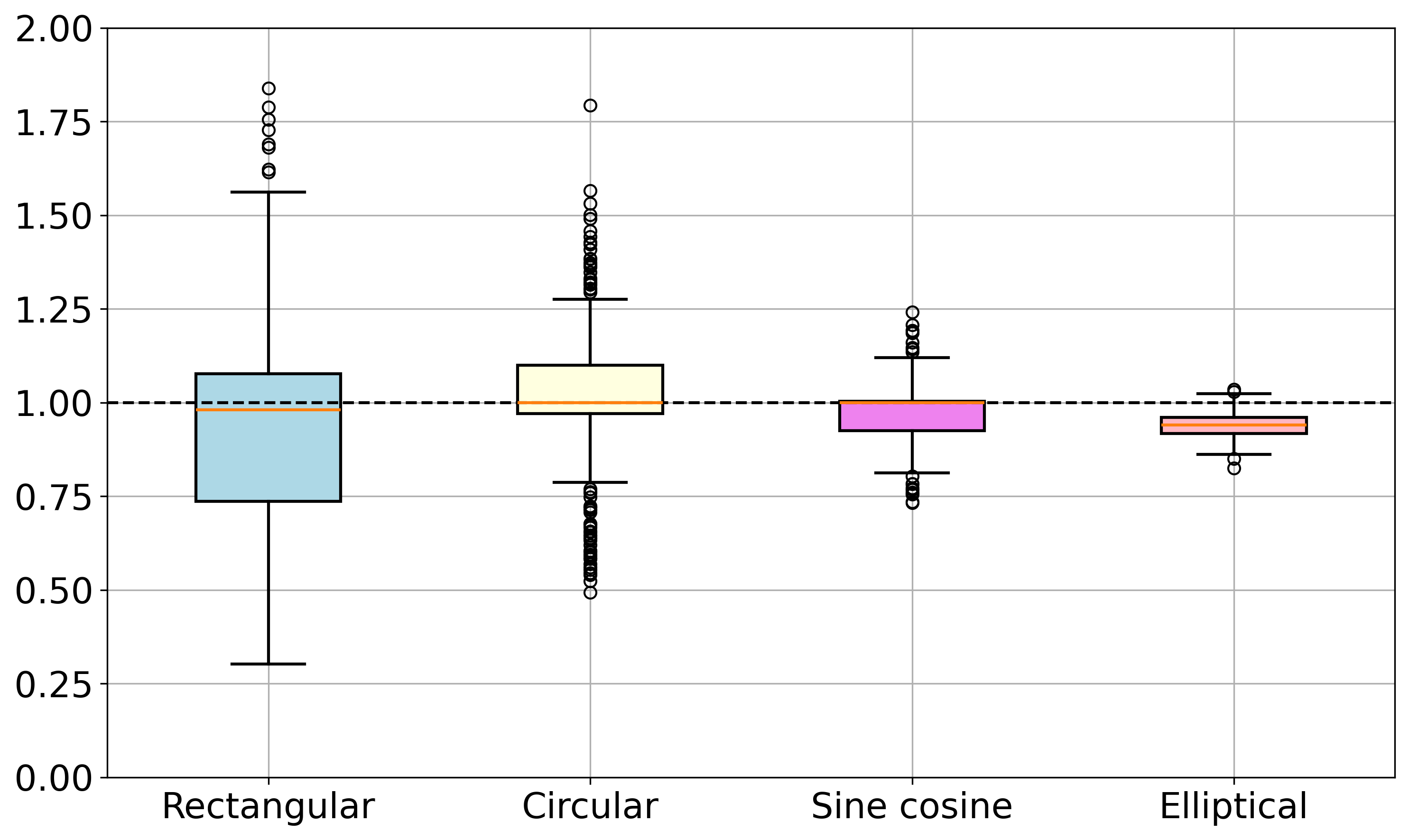}
        \caption{$\rho_{prun,semi}$ for Simulation A}
    \end{subfigure}
    \hspace{-0.2cm}
    \begin{subfigure}[b]{0.4\textwidth}
        \centering
        \includegraphics[width=\textwidth]{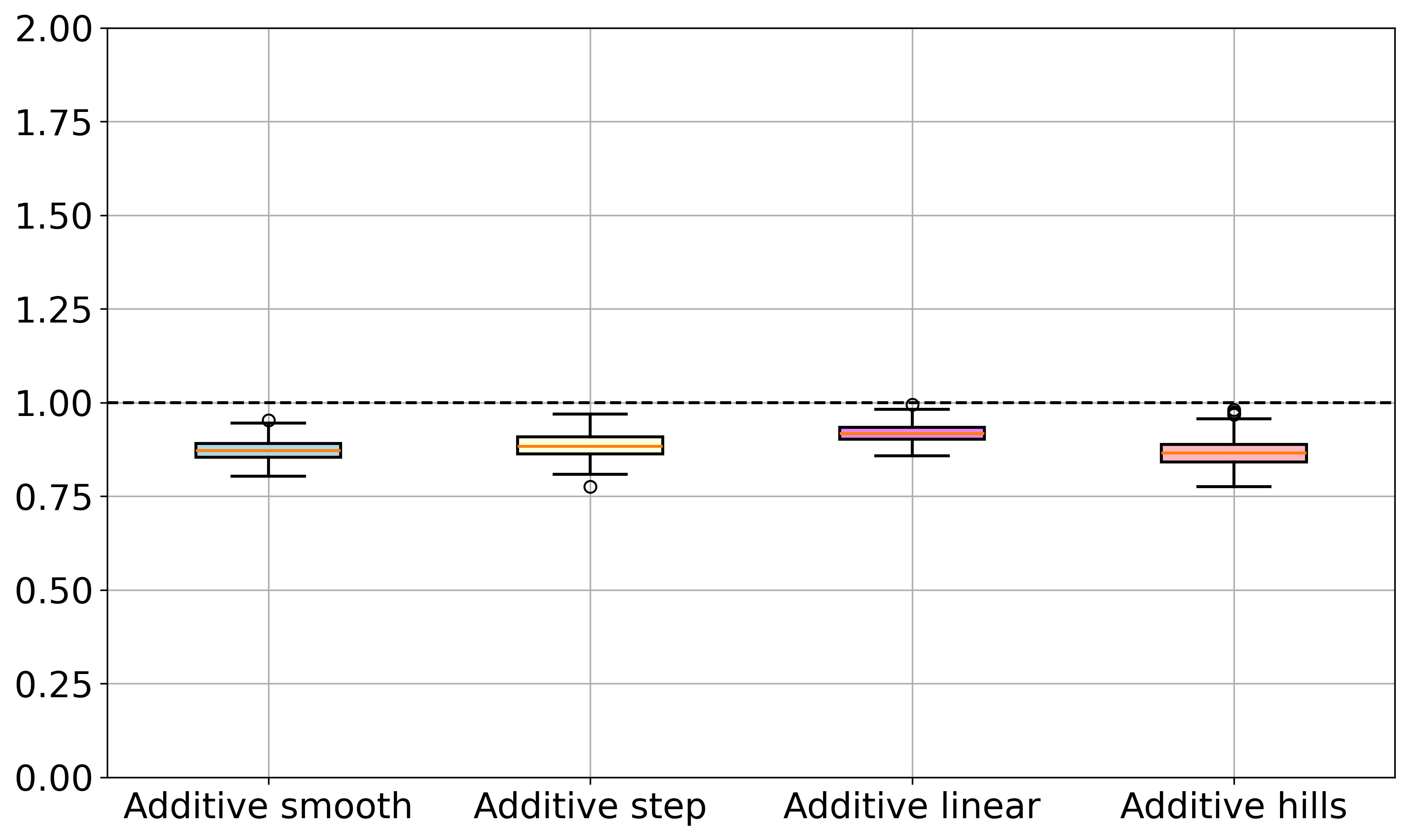}
        \caption{$\rho_{prun,semi}$ for Simulation B}
    \end{subfigure}
    \caption{Different oracle ratios: $\rho_{glob,semi}$ is global/semi-global (values smaller than one are in favor of global);
    $\rho_{prun,semi}$ is pruning/semi-global (values smaller than one are in favor of pruning).}
    \label{fig:ratios}
\end{figure}

We compare  the pruned, global, and semi-global oracles by the following ratios for each Monte Carlo run:
\begin{align*}
\textit{(global to semi-global)}   &\quad \rho_{glob,semi} = \min _{t \in [0, n]}\|\hat{F}_{t} - f \|_{n^{\prime}} / \min _{g \in \{0,\ldots,n \} } \|\hat{F}_{g} - f \|_{n^{\prime}},\\
%\textit{(pruning to global)} &\quad   \rho_{pg} = \min _{\lambda \in \Lambda}\|\widehat{F}_{\lambda} - f \|_{n^{\prime}} / \min _{t \in [0, n-1]}\|\widehat{F}_{t} - f \|_{n^{\prime}} \\
\textit{(pruning to semi-global)}   &\quad \rho_{prun,semi} = \min _{\lambda \in \Lambda_{prun}}\|\hat{F}_{\lambda} - f \|_{n^{\prime}} / \min _{g \in \{0,\ldots,n \} } \|\hat{F}_{g} - f \|_{n^{\prime}}.
\end{align*}
These oracle ratios are presented as boxplots in Figure \ref{fig:ratios}.
%Effect 1:All similar, except rectangular and circular
The first row shows that the oracle errors between the semi-global and global stopping methods are often very close.
For the {\it rectangular} and {\it circular} functions, however, the semi-global oracle performs significantly better. This is because the breadth-first algorithm splits all nodes at a given generation, which, in the case of these simple geometric boundaries, leads to several unnecessary splits in regions without signal. The global oracle thus suffers from the strongly varying local regularity, while the semi-global oracle much better adapts locally.

%Effect 2: Pruning vs ES
The second row of Figure \ref{fig:ratios} compares the pruning oracle to the semi-global oracle. The performance is quite similar, with slight advantages for pruning. For example, the oracle of the {\it elliptical} example is slightly more accurate for pruning, which is due to the bottom-up nature of the pruning. The first row of Figure \ref{fig:ratios} for the {\it elliptical} example shows that the semi-global oracle performs slightly better than the global oracle while having fewer terminal nodes, compare Table \ref{tab:times_nodes}. 

\bibliographystyle{agsm}
\bibliography{bibliography}

\end{document}